\documentclass[12pt, a4paper]{amsart}

\usepackage{amsmath,amsfonts,amsthm,amssymb,indentfirst,epic}

\usepackage{hyperref}

\newtheorem{thm}{Theorem}[section]
\newtheorem{cor}[thm]{Corollary}
\newtheorem{lem}[thm]{Lemma}

\newtheorem{prop}[thm]{Proposition}

\newtheorem{remark}[thm]{Remark}
\newtheorem{remarks}[thm]{Remarks}
\newtheorem{example}[thm]{Example}

\def\R {{\mathbb R}}

\def\N{{\mathbb N}}

\newcommand\Char{{\operatorname{Char}}}

\newcommand\trace{{\operatorname{trace}}}

\DeclareMathOperator{\M}{M}

\newcommand{\ben}{\begin{enumerate}}
\newcommand{\een}{\end{enumerate}}
\newcommand{\bit}{\begin{itemize}}
\newcommand{\eit}{\end{itemize}}

\newcommand\by{{\times}}
\DeclareMathOperator{\Cent}{Cent}
\renewcommand\H{{\mathbb{H}}}

\begin{document}
\title[]
{\large{Herstein's question about simple rings with involution
 }}
\author[]
{\small{Vered Moskowicz}}

\begin{abstract}
The aim of this paper is to try to answer Herstein's question concerning simple rings with involution, namely:
If $R$ is a simple ring with an involution of the first kind, with $dim_{Z(R)}R > 4$ and $\Char(Z(R))\neq 2$, is it true that $S^2=R$?
We shall see that in such a ring $R$, $R=S^3$.
We shall bring two possible criteria, each shows when $R=S^2$.
The first criterion: There exist $x,y \in S$ such that $xy-yx \neq 0$ and $xSy \subseteq S^2$ $\Leftrightarrow$ $S^2=R$.
The second criterion: There exist $x,y \in S$ such that $xy+yx \neq 0$ and $xKy \subseteq S^2$ $\Leftrightarrow$ $S^2=R$.
Actually, those results are true without any restriction on the dimension of $R$ over $Z(R)$.
In the special case of matrices (with the transpose involution and with the symplectic involution) over a field of characteristic not equal to $2$, it is not difficult to find, for example, $x,y \in S$ such that $xy-yx \neq 0$ and for every $s \in S$, $xsy \in S^2$. Therefore, proving Herstein's remark that for matrices the answer is known to be positive.
Similar results for $K^6$, $K^4$, $K+KSK$, $KS+K^2$, $SKS$ and $S^2K$ can also be found.

\end{abstract}

\maketitle

\section{Introduction}

Herstein has asked in his book \cite{her2} the following question:
Let $R$ be a simple ring with an involution of the first kind (namely, the center of $R$, $Z(R)$, is contained in the set of symmetric elements of $R$, $S$), with $dim_{Z(R)} > 4$ and $Z(R)$ is of characteristic not equal to $2$.
Is it true that $S^2=R$, where $S^2$ denotes the additive subgroup of $R$ generated by all $ab$ with $a,b \in S$?.
Herstein has remarked there, without a proof, that for matrices the answer is known to be positive.
Of course, if $R=S$ then $R=S^2$ (since $S \subseteq S^2$), so to avoid this trivial positive answer we will assume that $S \subsetneq R$.
Notice that if $S$ is a commutative set, then $S^2=S$. Therefore, if $S \subsetneq R$, then $S^2=S \subsetneq R$, a trivial negative answer.
As a first step towards an answer, we shall see that in such a ring $R$ (with $S$ not a commutative set) $R=S^3$, see Corollary \ref{simple S^3=R}.
As a second step, we shall see:
\bit
\item A first possible criterion for $R=S^2$: There exist $x,y \in S$ such that $xy-yx \neq 0$ and for every $s \in S$, $xsy \in S^2$ $\Leftrightarrow$ $S^2=R$, see Theorem \ref{first criterion}.
\item A second possible criterion for $R=S^2$: There exist $x,y \in S$ such that $xy+yx \neq 0$ and for every $k \in K$, $xky \in S^2$ $\Leftrightarrow$ $S^2=R$, see Theorem \ref{second criterion}.
\eit
Actually, those results are true without any restriction on the dimension of $R$ over $Z(R)$.
In the special case of matrices $\M_n(F)$ where $F$ is a field of characteristic not equal to $2$:
We shall suggest specific $x,y \in S$ such that $xy-yx \neq 0$ and for every $s \in S$, $xsy \in S^2$ for the transpose involution ($n \geq 2$) and for the symplectic involution ($2m=n \geq 4$), see Example \ref{section S^2: example matrices}.
Therefore, proving Herstein remark that for matrices the answer is positive.
So if one wishes to construct a negative answer to Herstein's question, one should construct a simple ring $R$ with an involution of the first kind with $dim_{Z(R)}R > 4$ and $Z(R)$ is of characteristic not equal to $2$ (and with $S \subsetneq R$), such that:
\bit
\item $S$ is a commutative set: Indeed, if $S$ is commutative, then Proposition \ref{S commutative}
shows that $S^2 \subsetneq R$. In fact, in this case Corollary \ref{cor of centS in Z} shows that $Z(R)=S$.
or
\item $S$ is not a commutative set and there exists an element $w \in S^3$ which is not in $S^2$ (since Corollary \ref{simple S^3=R} says that $S^3=R$. Even without Corollary \ref{simple S^3=R}, it is obvious that if there exists an element $w \in S^3$ which is not in $S^2$, then $S^2 \subsetneq S^3 \subseteq R$, so $S^2 \subsetneq R$).
\eit
Finally, similar results for other subsets, namely: $K^6$, $K^4$, $K+KSK$, $KS+K^2$, $SKS$ and $S^2K$ can be found in Theorem \ref{from Herstein 2.2 and 2.3}, Theorem \ref{R simple xy+yx neq 0 then R=K+KSK}, Theorem \ref{R simple two options for R=KS+K^2}, Corollary \ref{simple SKS=R} and Corollary \ref{simple S^2K=R}.

Some usual notations which will be used:
Let $R$ be any associative ring with $1$ (or any associative $F$-algebra with $1$, where $F$ is a field). $Z(R)$ will denote the center of $R$.
If $A,B$ are any subsets of $R$, then $A$ is a commutative set if for every $x,y \in A$, $xy=yx$.
$A$ is a skew-commutative set if for every $x,y \in A$, $xy=-yx$.
The centralizer of $A$, $Cent(A)$ is the set (one sees that it is actually a subring of $R$) of elements of $R$ which commute element-wise with the elements of $A$, $Cent(A)= \{r \in R| \forall a \in A ra=ar \}$.
$AB$ is the additive subgroup of $R$ generated by all products of the form $ab$ where $a \in A$ and $b \in B$.
($AB$ is the $F$-subspace of $R$ generated by all products of the form $ab$ where $a \in A$ and $b \in B$).
$A\circ B$ is the additive subgroup of $R$ generated by all products of the form $ab+ba$ where $a \in A$ and $B \in B$ ($\circ$ denotes the Jordan product $a\circ b=ab+ba$).
$\bar{A}$ is the subring ($F$-subalgebra) of $R$ generated by $A$.
If $R$ has an involution $*$, then $S$ will denote the set of symmetric elements of $R$ and $K$ will denote the set of skew-symmetric elements of $R$.
$T$ will denote the traces $T=\{r+r^* | r \in R \} \subseteq S$ and $T_0$ will denote the skew-traces $T_0=\{r-r^* | r \in R \} \subseteq K$.
In the matrix ring $\M_n(F)$, $e_{ij}$ will denote the $n \by n$ matrix having $1$ in the $ij$ component and $0$ otherwise.

We shall work in a context where $R$ will denote an associative $F$-algebra with $1$ and with an involution $*$ such that:
\bit
\item $F$ is a field of characteristic not equal to $2$. This implies that $R=S+K$, since $\forall r \in R$ $r=2^{-1}(r+r^*)+2^{-1}(r-r^*)$.
\item The involution $*$ is of the first kind, namely, $Z(R) \subseteq S$. This implies that $S \subseteq S^2 \subseteq S^3$ etc, since $1 \in Z(R) \subseteq S$.
\item $S \subsetneq R$. This assumption is made in order to dispose of the trivial case $S=R$, which implies $S^2 \supseteq S=R$, so Herstein's question has a positive (trivial) answer.
\eit
Notice that we have not made any assumptions on $dim_{Z(R)}R$ yet. Actually, we will not need this assumption. We shall say something about $dim_{Z(R)}R$ later on.

Next, we dispose of the following trivial case:
\begin{prop}\label{S commutative}
Let $R$ be an associative unital $F$-algebra with $\Char(F)\neq 2$. Let $*$ be an involution on $R$ of the first kind and $S \subsetneq R$. If $S$ is a commutative set, then $S^2\neq R$.
\end{prop}
\begin{proof}
By definition, $S^2=\{\sum_{1\leq i\leq n}a_ib_i |a_i,b_i \in S\}$.

Take $w \in S^2$. So, $w=\sum_{1\leq i\leq n}a_ib_i$ where $a_i,b_i \in S$. 
Then we have $w^*=(\sum a_ib_i)^*=\sum (a_ib_i)^*=\sum (b_i)^*(a_i)^*=\sum b_ia_i=\sum a_ib_i=w$, so $w \in S$, hence $S^2 \subseteq S$. 
But obviously, since $1 \in Z(R) \subseteq S$, $S \subseteq S^2$. 
Therefore, $S^2=S$.
Finally, using our assumption that $S$ is strictly contained in $R$, we get $S^2 \subsetneq R$.
\end{proof}

\section{Results about \texorpdfstring{$S^3, S^2$}{S3, S2}}
In view of Proposition \ref{S commutative}, if one wishes to find a positive answer to Herstein's question, one should assume that $S$ is not a commutative set. Hence, from now on we will usually deal with $S$ which is not a commutative set.

\subsection{Results about  \texorpdfstring{$S^3$}{S3}}
We give now two easy lemmas; both lemmas will be used in the proof of Theorem \ref{two sided in S^3}. 
Actually, in the proof of Theorem \ref{two sided in S^3} we will only use those two lemmas and not more.
Lemma \ref{first lemma (xy-yx)r in S^3} and Lemma \ref{second lemma r(xy-yx)t in S^3} do not say anything interesting if $S$ is a commutative set.
Similarly, when we consider $S^2$, Lemma \ref{first lemma xsy in S^2 then (xy-yx)r in S^2} and Lemma \ref{second lemma xsy in S^2 then r(xy-yx)t in S^2} do not say anything interesting if $S$ is a commutative set.
They are brought separately, each on its on right, and not just inside the proof of Theorem \ref{two sided in S^3}, because:
\bit
\item The first lemma (or more accurately, a special case of it) will be used directly in another result, namely (ii) of Proposition \ref{section S^3: xy-yx is invertible}.
\item The second lemma relies heavily on an idea of Herstein, more specifically, on the first part of the proof of Lemma 1.3 in \cite{her1}. We just change in Herstein's argument from $-$ to $+$, and add a $*$. This will be clear in the proof of the second lemma.
\eit

\begin{lem}\label{first lemma (xy-yx)r in S^3}
Let $R$ be an associative unital $F$-algebra with $\Char(F)\neq 2$. Let $*$ be an involution on $R$ of the first kind and $S \subsetneq R$. Let $x,y \in S$. Then for every $r \in R$, $(xy-yx)r \in S^3$. In other words, the right ideal $(xy-yx)R$ is contained in $S^3$.
\end{lem}

\begin{proof}
There are two options: \bit
\item $xy-yx=0$: Then, trivially, for every $r \in R$ $(xy-yx)r=0 \in S^3$.
\item $xy-yx\neq 0$: For every $r\in R$, $rxy+xyr^*=rxy+yxr^*-yxr^*+xyr^*$ (we just added and subtracted $yxr^*$).
Let $\alpha=\alpha(r)=rxy+xyr^*$, $\beta=\beta(r)=rxy+yxr^*$. We will just write $\alpha$ and $\beta$, instead of $\alpha(r)$ and $\beta(r)$. So for every $r \in R$, $\alpha=\beta+(xy-yx)r^*$.

Claim: $\alpha \in S^3$, $\beta \in S$.
Indeed, obviously $\beta \in S$, since $(\beta)^*=(rxy+yxr^*)*=yxr^*+rxy=\beta$.
As for $\alpha$: Of course, given $r \in R$, $r$ can be written as $r=s+k$ where $s \in S$ and $k \in K$. $\alpha=rxy+xyr^*=(s+k)xy+xy(s-k)=(sxy+xys)+(kxy-xyk)=(sxy+xys)+[(kx-xk)y-x(yk-ky)]$.
$sxy+xys \in S^3$ merely by definition of $S^3$ ($x,y,s \in S$). $(kx-xk)y-x(yk-ky) \in S^2$ since $kx-xk, y, x, yk-ky \in S$.

Using the claim we get: for every $r \in R$, $(xy-yx)r^*=\alpha-\beta \in S^3$ (remember that $S \subseteq S^2 \subseteq S^3$). But $R=R^*$ (since $r=(r^*)^*$), so for every $r \in R$, $(xy-yx)r \in S^3$.
(We could, of course, from the first place take $r^*xy+xyr=r^*xy+yxr-yxr+xyr$ instead of $rxy+xyr^*=rxy+yxr^*-yxr^*+xyr^*$, and get, without using $R=R^*$, that $(xy-yx)r \in S^3$).
Therefore, $(xy-yx)R \subseteq S^3$.

\eit
\end{proof}

If $S$ is a noncommutative set (in Lemma \ref{first lemma (xy-yx)r in S^3} we have not demanded that $S$ is a noncommutative set), then by definition exist $x,y \in S$ such that $xy-yx \neq 0$, therefore $S^3$ contains a non-zero right ideal of $R$, namely $(xy-yx)R$ where $x,y \in S$ with $xy-yx\neq 0$ (clearly, $0 \neq (xy-yx)1 \in (xy-yx)R$).

\begin{lem}\label{second lemma r(xy-yx)t in S^3}
Let $R$ be an associative unital $F$-algebra with $\Char(F)\neq 2$. Let $*$ be an involution on $R$ of the first kind and $S \subsetneq R$. Let $x,y \in S$. Then for every $r,u \in R$, $u(xy-yx)r \in S^3$. In other words, the two-sided ideal $R(xy-yx)R$ is contained in $S^3$.
\end{lem}

\begin{proof}
There are two options: \bit
\item $xy-yx=0$: Then, trivially, for every $r,u \in R$ $u(xy-yx)r=0 \in S^3$.
\item $xy-yx\neq 0$:

Claim: for every $u \in R$ and $w \in S^3$, $wu+u^*w \in S^3$.
It is enough to show that for every $a,b,c \in S$, $abcu+u^*abc \in S^3$ (indeed, if for every $a_i,b_i,c_i \in S$, $a_ib_ic_iu+u^*a_ib_ic_i \in S^3$, then taking $S^3 \ni w=\sum a_ib_ic_i$, we get $wu+u^*w=(\sum a_ib_ic_i)u+u^*(\sum a_ib_ic_i)=\sum (a_ib_ic_iu+u^*a_ib_ic_i)$, so since each $a_ib_ic_iu+u^*a_ib_ic_i \in S^3$, we get $wu+u^*w=\ldots=\sum (a_ib_ic_iu+u^*a_ib_ic_i) \in S^3$).
Now, $abcu+u^*abc=ab(cu+u^*c)+(u^*a+au)bc-abu^*c-aubc=ab(cu+u^*c)+(u^*a+au)bc-a(bu^*+ub)c$. Clearly, $cu+u^*c, u^*a+au, bu^*+ub \in S$, concluding that $abcu+u^*abc=\ldots=ab(cu+u^*c)+(u^*a+au)bc-a(bu^*+ub)c \in S^3$.
So we have proved the claim that for every $u \in R$ and $w \in S^3$, $wu+u^*w \in S^3$.

Next, the above lemma, Lemma \ref{first lemma (xy-yx)r in S^3}, says that for every $r \in R$, $(xy-yx)r \in S^3$, so from the claim just seen (with $w=(xy-yx)r$), for every $u \in R$ and for every $r \in R$, $((xy-yx)r)u+u^*(xy-yx)r \in S^3$.
Use again the above lemma to get that $((xy-yx)r)u=(xy-yx)ru \in S^3$.
Therefore, for every $u \in R$ and for every $r \in R$, $u^*(xy-yx)r in S^3$. Obviously, since $R=R^*$ we get that
for every $u \in R$ and for every $r \in R$, $u(xy-yx)r in S^3$. So $R(xy-yx)R \subseteq S^3$.

\eit
\end{proof}

If $S$ is a noncommutative set (in Lemma \ref{second lemma r(xy-yx)t in S^3} we have not demanded that $S$ is a noncommutative set), then by definition exist $x,y \in S$ such that $xy-yx \neq 0$, therefore $S^3$ contains a non-zero two-sided ideal of $R$, namely $R(xy-yx)R$ where $x,y \in S$ with $xy-yx\neq 0$ (clearly, $0 neq 1(xy-yx)1 \in R(xy-yx)R$).
We proceed to our theorem.

\begin{thm}\label{two sided in S^3}
Let $R$ be an associative unital $F$-algebra with $\Char(F)\neq 2$. Let $*$ be an involution on $R$ of the first kind and $S \subsetneq R$. Assume $S$ is not a commutative set. Then there exists a non-zero two-sided ideal of $R$ which is contained in $S^3$.
\end{thm}

\begin{proof}
Follows immediately form Lemma \ref{second lemma r(xy-yx)t in S^3}, as was explained immediately after it.
\end{proof}

Theorem \ref{two sided in S^3} immediately implies the following:

\begin{cor}[First step towards an answer to Herstein's question]\label{simple S^3=R}
Let $R$ be a simple associative unital $F$-algebra with $\Char(F)\neq 2$. Let $*$ be an involution on $R$ of the first kind and $S \subsetneq R$. Assume $S$ is not a commutative set. Then $S^3=R$.
\end{cor}

\begin{proof}
Follows at once from Theorem \ref{two sided in S^3}.
\end{proof}

\begin{remark}
Let $R$ be as in Corollary \ref{simple S^3=R}.
\bit
\item If $dim_F(R)=n \infty$, then $n > dim_F(S)\geq n^{1/3}$: 

Indeed, $dim_F(R)= dim_F(S^3) \leq (dim_F(S))^3$ (the first equality is true since $R=S^3$). 
Therefore, $n \leq (dim_F(S))^3$, so $dim_F(S)\geq n^{1/3}$.
\item If $dim_F(R)=\infty$, then $dim_F(S)=\infty$: 

Otherwise, $dim_F(S)=m < \infty$. But then $\infty= dim_F(R)= dim_F(S^3) \leq (dim_F(S))^3=m^3$, a contradiction.
\eit
\end{remark}

\begin{example}[Major example: Matrices over a field]\label{section S^3: example matrices} 

Let $R=\M_n(F)$, where $F$ is a field of characteristic not equal to $2$ and $n>1$ (of course, if $n=1$ then $R=S$ and we have already disposed of that case at the beginning).
It is well known that $R$ is simple and that there exist two involutions on $R$: the transpose involution and the symplectic involution (remember that the symplectic involution only exists for $n=2m$).
See Rowen's book for a discussion about the two involutions on $\M_n(F)$ \cite[page 43]{rowen}.
Corollary \ref{simple S^3=R} shows that:
\bit
\item The transpose involution: For $n\geq 2$, $S^3=R$ (it is easily seen that $S$ is not a commutative set).
\item The symplectic involution: For $n=2m \geq 4$, $S^3=R$ (it is easily seen that $S$ is not a commutative set).
Notice that when $n=2$ one cannot use Corollary \ref{simple S^3=R}, since in that corollary $S$ is assumed to be a noncommutative set, but in $\M_2(F)$, $S\cong F$, so $S$ is a commutative set (moreover, $S$ is a field, not just a set). Clearly, in $\M_2(F)$, $S^3=S^2=S \cong F \subsetneq \M_2(F)$.
\eit
\end{example}

\subsubsection{$S^3=R$ without simplicity}

If one prefers to dismiss of the condition that $R$ is simple, but still have $S^3=R$, then we offer the following proposition.
It demands a stronger condition then noncommutativity of $S$ (for example, in (ii) we demand that there exist $x,y \in S$ such that $xy-yx$ is right invertible).

\begin{prop}\label{section S^3: xy-yx is invertible}
Let $R$ be an associative unital $F$-algebra with $\Char(F)\neq 2$. Let $*$ be an involution on $R$ of the first kind and $S \subsetneq R$. Assume $S$ is not a commutative set. Then, \bit
\item [(i)] If there exist $x,y \in S$ such that $xy-yx \neq 0$ and such that $1$ can be written as $\sum_{1 \leq i \leq m}u_i(xy-yx)v_i$ where $u_i,v_i \in R$, then $S^3=R$.
\item [(ii)] If a stronger condition then the condition in (i) is satisfied, namely, if there exist $x,y \in S$ such that $xy-yx$ is right invertible, then (in addition to $S^3=R$) each element of $R$ can be written in the form $\lambda+ \mu + \nu$, where $\lambda \in S$, $\mu =a_1b_1+a_2b_2$, $\nu=\tilde{a_1}\tilde{b_1}\tilde{c_1}+\tilde{a_2}\tilde{b_2}\tilde{c_2}$, with $a_1 ,b_1, a_2, b_2, \tilde{a_1}, \tilde{b_1}, \tilde{c_1}, \tilde{a_2}, \tilde{b_2}, \tilde{c_2} \in S$.
\eit
\end{prop}

\begin{remarks}
\bit
\item Notice that we could drop the condition "Assume $S$ is not a commutative set", since afterwards in (i) and in (ii) we already demand that there exist $x,y \in S$ such that $xy-yx \neq 0$.
\item Of course (ii) is stronger then (i): Let $z \in R$ such that $(xy-yx)z=1$. Then (i) is indeed satisfied by taking, for example, $u_1=1, v_1=z$.
\item (ii) of this proposition says that there is an upper bound on the number of monomials of length $3$ in the representation of every element $r \in R$, namely: $5$ monomials at most are needed.
(We consider $a \in S$ as having length $3$, for example $a=a11$. Similarly $ab$ is of length $3$, $ab=ab1$ where $a,b \in S$).
\item For the proof of (ii) only the first lemma, Lemma \ref{first lemma (xy-yx)r in S^3} is needed.
However, for the proof (i) the second lemma, Lemma \ref{second lemma r(xy-yx)t in S^3} is needed.
\eit
\end{remarks}

\begin{proof}
\bit
\item [(i)] Write $1=\sum_{1 \leq i \leq m}u_i(xy-yx)v_i$ where $u_i,v_i \in R$ and $x,y \in S$ such that $xy-yx \neq 0$. This of course implies that the non-zero two-sided ideal $R(xy-yx)R$ equals $R$. But from Lemma \ref{second lemma r(xy-yx)t in S^3}, $R(xy-yx)R$ is contained in $S^3$. Therefore, $S^3=R$.
\item [(ii)] By assumption there exists $x,y \in S$ such that $xy-yx$ is right invertible, so there exists $z \in R$ such that $(xy-yx)z=1$.
For every $r\in R$, $rxy+xyr^*=rxy+yxr^*-yxr^*+xyr^*$. Let $\alpha=\alpha(r)=rxy+xyr^*$, $\beta=\beta(r)=rxy+yxr^*$.  So for every $r \in R$, $(xy-yx)r^*=\alpha(r)-\beta(r)$.
We have seen in the proof of Lemma \ref{first lemma (xy-yx)r in S^3} that $\beta(r) \in S$ and $\alpha(r) \in S^3$.
More elaborately for $\alpha(r)$, we have seen that if $r=s+k$ with $s \in S$ and $k \in K$:
$\alpha(r)=rxy+xyr^*=\ldots=(sxy+xys)+[(kx-xk)y-x(yk-ky)]$.

So denote: $\lambda(r)=-\beta(r)=-(rxy+yxr^*)$, $\mu(r)=(kx-xk)y-x(yk-ky)$ and $/nu(r)=sxy+xys$.
Hence, $(xy-yx)r^*=-\beta(r)+\alpha(r)=\lambda(r)+\mu(r)+\nu(r)$.
Define $a_1(r)=kx-xk, b_1(r)=y, a_2(r)=-x, b_2(r)=yk-ky, \tilde{a_1}(r)=s, \tilde{b_1}(r)=x, \tilde{c_1}(r)=y, \tilde{a_2}(r)=x, \tilde{b_2}(r)=y, \tilde{c_2}(r)=s$.
Therefore, $(xy-yx)r^*=-\beta(r)+\alpha(r)=\lambda(r)+\mu(r)+\nu(r)=\lambda(r)+a_1(r)b_1(r)+a_2(r)b_2(r)+ \tilde{a_1}(r)\tilde{b_1}(r)\tilde{c_1}(r)+\tilde{a_2}(r)\tilde{b_2}(r)\tilde{c_2}(r)$.

Now, for $r \in R$, $r=1r=((xy-yx)z)r=(xy-yx)(zr)$ (here we use the right invertibility of $xy-yx$).

In our notations,

$r=(xy-yx)(zr)=-\beta((zr)^*)+\alpha((zr)^*)=\lambda((zr)^*)+\mu((zr)^*)+\nu((zr)^*)=
\lambda(zr)^*+a_1((zr)^*)b_1((zr)^*)+a_2((zr)^*)b_2((zr)^*)+\tilde{a_1}((zr)^*)\tilde{b_1}((zr)^*)\tilde{c_1}((zr)^*)+
\tilde{a_2}((zr)^*)\tilde{b_2}((zr)^*)\tilde{c_2}((zr)^*)$.

So each element $r \in R$ can be written in the requested form.
\eit
\end{proof}

Notice that Corollary \ref{simple S^3=R} just says that in such a simple ring, each element $r \in R$ is in $S^3$, hence can be written in the form $r=\sum_{1\leq i \leq m}a_ib_ic_i$ with $a_i,b_i,c_i \in S$.
However, it does not tell if there is any bound on $m$.
But if $xy-yx$ is (not just non-zero, but also) right invertible, then there is a bound, namely $5$, as (ii) of Proposition \ref{section S^3: xy-yx is invertible} shows.
However, there are cases where a better bound then $5$ can be given, for example, if $R=\M_2(F)$ with the transpose involution, then $2$ is a bound.
This can be seen, for example, as follows: Let $R \ni r=ae_11+be_{12}+ce_{21}+de_{22}$.

$r=(1/2)(r+r^*)+(1/2)(r-r^*)=
[ae_{11}+(1/2)(b+c)e_{12}+(1/2)(b+c)e_{21}+de_{22}]+[0e_{11}+(1/2)(b-c)e_{12}+(1/2)(c-b)e_{21}+0e_{22}]$.

The first term, $[ae_{11}+(1/2)(b+c)e_{12}+(1/2)(b+c)e_{21}+de_{22}]$ is in $S$.
The second term, $[0e_{11}+(1/2)(b-c)e_{12}+(1/2)(c-b)e_{21}+0e_{22}]$ is a product of two elements of $S$: the diagonal matrix $(1/2)(b-c)e_{11}+(1/2)(c-b)e_{22}$ and $e_{12}+e_{21}$.

\begin{example}\label{section S^3 without simple: example matrices}

Let $R=\M_n(F)$, where $F$ is a field of characteristic not equal to $2$, $n>1$.
(ii) of Proposition \ref{section S^3: xy-yx is invertible} shows that:
\bit
\item The transpose involution: Let $n=2l$. Notice that, in contrast to Example \ref{section S^3: example matrices}, we now demand that $n$ will be even.
Then each element of $R$ can be written in the form $\lambda+ \mu + \nu$,
where $\lambda \in S$, $\mu =a_1b_1+a_2b_2$, $\nu=\tilde{a_1}\tilde{b_1}\tilde{c_1}+\tilde{a_2}\tilde{b_2}\tilde{c_2}$,
with $a_1 ,b_1, a_2, b_2, \tilde{a_1}, \tilde{b_1}, \tilde{c_1}, \tilde{a_2}, \tilde{b_2}, \tilde{c_2} \in S$
(By $S$ we mean the symmetric elements wrt the transpose involution).
Indeed, one can take $x=(e_{11}-e_{22})+(e_{33}-e_{44})+\ldots+(e_{n-1,n-1}-e_{nn})$, $y=(e_{12}+e_{21})+(e_{34}+e_{43})+\ldots+(e_{n-1,n}+e_{n,n-1})$. Obviously, $x,y \in S$.
A computation shows that $xy=(e_{12}-e_{21})+(e_{34}-e_{43})+\ldots+(e_{n-1,n}-e_{n,n-1})$ and $yx=-xy$, so $xy-yx=2xy=2[(e_{12}-e_{21})+(e_{34}-e_{43})+\ldots+(e_{n-1,n}-e_{n,n-1})]$.
One sees that $xy-yx$ is invertible, so we can apply (ii) of Proposition \ref{section S^3: xy-yx is invertible}.

A word of caution: If, for example, $n=3$, then do not exist $x,y \in S$ such that $xy-yx$ is invertible.
Reason: $xy-yx \in K$, so $xy-yx=ae_{12}-ae_{21}+be_{13}-be_{31}+ce_{23}-ce_{32}$ for some $a,b,c \in F$.
A direct computation shows that the determinant of $xy-yx$ is $0$ (for example, working in the first column yields: $a(bc)-b(ac)$ which is, of course, $0$).

This just shows that one cannot use (ii) of Proposition \ref{section S^3: xy-yx is invertible} to give an upper bound for the number of summands $a_ib_ic_i$ ($a_i,b_i,c_i \in S$) in the representation of $S^3=R \ni r=\sum_{1 \leq i \leq m(r)} a_ib_ic_i$ ($m(r)$ depends on $r$).
The reader is invited to check and see what happens if $n \in \{5,7,9,\ldots\}$.
\item The symplectic involution: Let $n=2m$, with $m=2l$. Notice that, in contrast to Example \ref{section S^3: example matrices}, we now demand that $m$ will be even.
Then each element of $R$ can be written in the form $\lambda+ \mu + \nu$,
where $\lambda \in S$, $\mu =a_1b_1+a_2b_2$, $\nu=\tilde{a_1}\tilde{b_1}\tilde{c_1}+\tilde{a_2}\tilde{b_2}\tilde{c_2}$, with $a_1 ,b_1, a_2, b_2, \tilde{a_1}, \tilde{b_1}, \tilde{c_1}, \tilde{a_2}, \tilde{b_2}, \tilde{c_2} \in S$
(By $S$ we mean the symmetric elements wrt the symplectic involution).
Indeed, one can take $x$ as the blocks matrix with two equal blocks, each of the two blocks is an $m \by m$ ($m=2l$) matrix of the form: $(e_{11}-e_{22})+(e_{33}-e_{44})+\ldots+(e_{m-1,m-1}-e_{mm})$.
And one can take $y$ as the blocks matrix with two equal blocks, each of the two blocks is an $m \by m$ ($m=2l$) matrix of the form: $(e_{12}+e_{21})+(e_{34}+e_{43})+\ldots+(e_{m-1,m}+e_{m,m-1})$. Obviously, $x,y \in S$.
A computation shows that $xy=(e_{12}-e_{21})+(e_{34}-e_{43})+\ldots+(e_{n-1,n}-e_{n,n-1})$ and $yx=-xy$, so $xy-yx=2xy=2[(e_{12}-e_{21})+(e_{34}-e_{43})+\ldots+(e_{n-1,n}-e_{n,n-1})]$.
One sees that $xy-yx$ is invertible, so we can apply (ii) of Proposition \ref{section S^3: xy-yx is invertible}.
The reader is invited to check and see what happens if $m \in \{5,7,9,\ldots\}$.
\eit
\end{example}


\subsubsection{A remark about the dimension of $R$}

For the next few results we will consider only a simple ring $R$.
We would like to remark about the dimension of a simple $R$ over $Z(R)$ (remember that the center of a simple ring is a field).

\begin{thm}[A theorem of Herstein]\label{centS in Z} 
Let $R$ be a simple associative unital $F$-algebra ($F$ is of any characteristic) with an involution $*$ (not necessarily of the first kind) and assume $dim_{Z(R)}R > 4$.
Then $\Cent(S)\subseteq Z(R)$ (therefore, $\Cent(S)=Z(R)$).
\end{thm}

Notice that if $S=R$, then the theorem is trivial, since $\Cent(S)=\Cent(R)=Z(R)$. Anyway, remember that at the beginning we have mentioned that we shall be interested in $S \subsetneq R$ (since if $S=R$, then $S^2\supseteq S=R$).

\begin{proof}
This is Theorem 1.7 in Herstein's book \cite{her1}.
\end{proof}

As a corollary to the above theorem we have:

\begin{cor}\label{cor of centS in Z}
Let $R$ be a simple associative unital $F$-algebra with $\Char(F)\neq 2$. Let $*$ be an involution on $R$ of the first kind and $S \subsetneq R$. Also assume that $dim_{Z(R)}R > 4$.
Then, $S$ is a commutative set $\Leftrightarrow$ $Z(R)=S$.
\end{cor}

Notice that apriori, if $Z(R) \subsetneq S$, then $S$ could be commutative or not, but this corollary says that $S$ must be noncommutative.

\begin{proof}
Of course if $S=Z(R)$ then $S$ is commutative (moreover, $S$ is not just a commutative set, but a field).
As for the other direction, assume that $S$ is a commutative set. Then $S \subseteq \Cent(S)$. Now use Herstein's theorem \ref{centS in Z}, so $S \subseteq \Cent(S) \subseteq Z(R)$. Finally, since $*$ is of the first kind (namely $Z(R) \subseteq S$), we get $S \subseteq \Cent(S) \subseteq Z(R) \subseteq S$. Therefore, $S=Z(R)$.
\end{proof}

In view of this, one has:

\begin{thm}\label{simple with dim > 4}
Let $R$ be a simple associative unital $F$-algebra with $\Char(F)\neq 2$. Let $*$ be an involution on $R$ of the first kind such that $Z(R) \subsetneq S$ (and not just $Z(R) \subseteq S$) and $S \subsetneq R$. Assume that $dim_{Z(R)}R > 4$. Then $S^3=R$.
\end{thm}

\begin{proof}
By assumption, $Z(R) \subsetneq S$, so from Corollary \ref{cor of centS in Z} we get that $S$ is noncommutative.
Now Corollary \ref{simple S^3=R} implies that $S^3=R$.
\end{proof}

Notice that one can apply Corollary \ref{simple S^3=R} to $\M_2(F)$ with the transpose involution, where $\Char(F)\neq 2$ (there is no restriction on the dimention of $R$ over $Z(R)$), while Theorem \ref{simple with dim > 4} cannot be applied, since $dim_{Z(\M_2(F))}\M_2(F) = 4$.


\subsection{Results about  \texorpdfstring{$S^2$}{S2}}

However, Herstein's question dealt with $S^2$ and not $S^3$.
Therefore, we shall bring two possible criteria, each shows when $R=S^2$ ($R$ is a simple associative unital $F$-algebra with $\Char(F)\neq 2$. $*$ is an involution on $R$ of the first kind and $S \subsetneq R$. $S$ is not a commutative set).
The first criterion says: There exist $x,y \in S$ such that $xy-yx \neq 0$ and for every $s \in S$, $xsy \in S^2$ if and only if $S^2=R$, see Theorem \ref{first criterion}.
The second criterion says: There exist $x,y \in S$ such that $xy+yx \neq 0$ and for every $k \in K$, $xky \in S^2$ if and only if $S^2=R$, see Theorem \ref{second criterion}.
Those results are true without any restriction on the dimension of $R$ over $Z(R)$.

\subsubsection{First criterion}

We start with two lemmas, similar to Lemma \ref{first lemma (xy-yx)r in S^3} and Lemma \ref{second lemma r(xy-yx)t in S^3}.

\begin{lem}\label{first lemma xsy in S^2 then (xy-yx)r in S^2}
Let $R$ be an associative unital $F$-algebra with $\Char(F)\neq 2$. Let $*$ be an involution on $R$ of the first kind and $S \subsetneq R$. Let $x,y \in S$ such that for every $s \in S$, $xsy \in S^2$.
Then for every $r \in R$, $(xy-yx)r \in S^2$. In other words, the right ideal $(xy-yx)R$ is contained in $S^2$.
\end{lem}

\begin{proof}
There are two options: \bit
\item $xy-yx=0$: Then, trivially, for every $r \in R$ $(xy-yx)r=0 \in S^2$.
\item $xy-yx\neq 0$: For every $r\in R$, $rxy+xyr^*=rxy+yxr^*-yxr^*+xyr^*$ (we just added and subtracted $yxr^*$).
Let $\alpha=\alpha(r)=rxy+xyr^*$, $\beta=\beta(r)=rxy+yxr^*$. We will just write $\alpha$ and $\beta$, instead of $\alpha(r)$ and $\beta(r)$. So for every $r \in R$, $\alpha=\beta+(xy-yx)r^*$.

Claim: $\alpha \in S^2$, $\beta \in S$.
Indeed, obviously $\beta \in S$, since $(\beta)^*=(rxy+yxr^*)*=yxr^*+rxy=\beta$.
As for $\alpha$: Of course, given $r \in R$, $r$ can be written as $r=s+k$ where $s \in S$ and $k \in K$. $\alpha=rxy+xyr^*=(s+k)xy+xy(s-k)=(sxy+xys)+(kxy-xyk)=(sxy+xys)+[(kx-xk)y-x(yk-ky)]$.

We show now that $sxy+xys \in S^2$: $sxy+xys=(sx+xs)y+x(ys+sy)-2xsy$. But, $sx+sx \in S, y \in S, x \in S, ys+sy \in S$, and $-2xsy \in S^2$ (just by our assumption that for every $\tilde{s} \in S$, $x\tilde{s}y \in S^2$), concluding that $sxy+xys \in S^2$.
$(kx-xk)y-x(yk-ky) \in S^2$ since $kx-xk, y, x, yk-ky \in S$.

Using the claim we get: for every $r \in R$, $(xy-yx)r^*=\alpha-\beta \in S^2$ (remember that $S \subseteq S^2$). But $R=R^*$ (since $r=(r^*)^*$), so for every $r \in R$, $(xy-yx)r \in S^2$.
(We could, of course, from the first place take $r^*xy+xyr=r^*xy+yxr-yxr+xyr$ instead of $rxy+xyr^*=rxy+yxr^*-yxr^*+xyr^*$, and get, without using $R=R^*$, that $(xy-yx)r \in S^2$).
Therefore, $(xy-yx)R \subseteq S^2$.
\eit
\end{proof}

If $S$ is a noncommutative set (in Lemma \ref{first lemma xsy in S^2 then (xy-yx)r in S^2} we have not demanded that $S$ is a noncommutative set), then by definition exist $x,y \in S$ such that $xy-yx \neq 0$.
However, in contrast to what we have seen in the above section about $S^3$ (noncommutativity of $S$ implies that $S^3$ contains a non-zero right ideal of $R$), noncommutativity of $S$ does not imply that $S^2$ contains a non-zero right ideal of $R$.
But, if those $x,y \in S$ such that $xy-yx \neq 0$ also satisfy $xSy \subseteq S^2$ (this means that for every $s \in S$, $xsy \in S^2$), then $S^2$ contains a non-zero right ideal of $R$ (namely, $(xy-yx)R$ where $x,y \in S$ with $xy-yx\neq 0$ and $xSy \subseteq S^2$).

\begin{lem}\label{second lemma xsy in S^2 then r(xy-yx)t in S^2}
Let $R$ be an associative unital $F$-algebra with $\Char(F)\neq 2$. Let $*$ be an involution on $R$ of the first kind and $S \subsetneq R$. Let $x,y \in S$ such that for every $s \in S$, $xsy \in S^2$. Then for every $r,u \in R$, $u(xy-yx)r \in S^2$. In other words, the two-sided ideal $R(xy-yx)R$ is contained in $S^2$.
\end{lem}

\begin{proof}
There are two options: \bit
\item $xy-yx=0$: Then, trivially, for every $r,u \in R$ $u(xy-yx)r=0 \in S^2$.
\item $xy-yx\neq 0$:

Claim: for every $u \in R$ and $w \in S^2$, $wu-uw \in S^2$ (in other words, $S^2$ is a Lie ideal of $R$).
It is enough to show that for every $a,b \in S$, $abu-uab \in S^2$ (indeed, if for every $a_i,b_i \in S$, $a_ib_iu-ua_ib_i \in S^2$, then taking $S^2 \ni w=\sum a_ib_i$, we get $wu-uw=(\sum a_ib_i)u-u(\sum a_ib_i)=\sum (a_ib_iu-ua_ib_i)$, so since each $a_ib_iu-ua_ib_i \in S^2$, we get $wu-uw=\ldots=\sum (a_ib_iu-ua_ib_i) \in S^2$).
Now, $abu-uab=a(bu+u^*b)-(ua+au^*)b$. Clearly, $bu+u^*b, ua+au^* \in S$, concluding that $abu-uab \in S^2$.
So we have proved the claim that for every $u \in R$ and $w \in S^2$, $wu-uw \in S^2$.

Next, the above lemma, Lemma \ref{first lemma xsy in S^2 then (xy-yx)r in S^2}, says that for every $r \in R$, $(xy-yx)r \in S^2$, so from the claim just seen (with $w=(xy-yx)r$), for every $u \in R$ and for every $r \in R$, $((xy-yx)r)u-u((xy-yx)r) \in S^2$.
Use again the above lemma to get that (for every $u \in R$ and for every $r \in R$) $((xy-yx)r)u=(xy-yx)ru \in S^2$.
Therefore, for every $u \in R$ and for every $r \in R$, $-u((xy-yx)r) \in S^2$. So $R(xy-yx)R \subseteq S^2$.
\eit
\end{proof}

If $S$ is a noncommutative set (in Lemma \ref{second lemma xsy in S^2 then r(xy-yx)t in S^2} we have not demanded that $S$ is a noncommutative set), then by definition exist $x,y \in S$ such that $xy-yx \neq 0$.
However, in contrast to what we have seen in the above section about $S^3$ (noncommutativity of $S$ implies that $S^3$ contains a non-zero two-sided ideal of $R$), noncommutativity of $S$ does not imply that $S^2$ contains a non-zero two-sided ideal of $R$.
But, if those $x,y \in S$ such that $xy-yx \neq 0$ also satisfy $xSy \subseteq S^2$, then $S^2$ contains a non-zero two-sided ideal of $R$ (namely, $R(xy-yx)R$ where $x,y \in S$ with $xy-yx\neq 0$ and $xSy \subseteq S^2$).
Next, similarly to Theorem \ref{two sided in S^3} which dealt with $S^3$, we have the following theorem which deals with $S^2$.

\begin{thm}\label{xsy in S^2 then two sided in S^2}
Let $R$ be an associative unital $F$-algebra with $\Char(F)\neq 2$. Let $*$ be an involution on $R$ of the first kind and $S \subsetneq R$. Assume $S$ is not a commutative set. Also assume that there exist $x,y \in S$ such that $xy-yx \neq 0$ and for every $s \in S$, $xsy \in S^2$. Then there exists a non-zero two-sided ideal of $R$ which is contained in $S^2$.
\end{thm}

\begin{remark}
Of course the assumption that $S$ is not a commutative set implies that there exists $x,y \in S$ such that $xy-yx \neq 0$. Here we demand more; we demand to find such non commute $x,y \in S$ with $xsy \in S^2$ for every $s \in S$ (by definition of $S^3$, $xsy \in S^3$ for every $s \in S$).
\end{remark}

\begin{proof}
Follows immediately form Lemma \ref{second lemma xsy in S^2 then r(xy-yx)t in S^2}, as was explained immediately after it.
\end{proof}

Similarly to Theorem \ref{two sided in S^3} and its corollary, we have as a corollary:

\begin{cor}[Second step towards an answer to Herstein's question- A first criterion]\label{xsy in S^2 and simple then S^2=R} 
Let $R$ be a simple associative unital $F$-algebra with $\Char(F)\neq 2$. Let $*$ be an involution on $R$ of the first kind and $S \subsetneq R$. Assume $S$ is not a commutative set. Also assume that there exist $x,y \in S$ such that $xy-yx \neq 0$ and for every $s \in S$, $xsy \in S^2$. Then $S^2=R$.
\end{cor}

\begin{proof}
Follows at once from Theorem \ref{xsy in S^2 then two sided in S^2}.
\end{proof}

\begin{remark}
Let $R$ be as in Corollary \ref{xsy in S^2 and simple then S^2=R}.
If $dim_F(R)=n \infty$, then $n > dim_F(S)\geq n^{1/2}$: Indeed, $dim_F(R)= dim_F(S^2) \leq (dim_F(S))^2$ (the first equality is true since $R=S^2$). Therefore, $n \leq (dim_F(S))^2$, so $dim_F(S)\geq n^{1/2}$.
\end{remark}


If one wonders if the condition that there exist $x,y \in S$ such that $xy-yx \neq 0$ and for every $s \in S$, $xsy \in S^2$ is satisfied by matrices, then the answer is yes, as the next example shows.

\begin{example}[Major example: Matrices over a field]\label{section S^2: example matrices} 

Let $R=\M_n(F)$, where $F$ is a field of characteristic not equal to $2$ and $n>1$.
Corollary \ref{xsy in S^2 and simple then S^2=R} can be applied for both the transpose involution and the symplectic involution on $\M_n(F)$, thus showing that in each of these two cases $S^2=R$. Indeed;\bit
\item The transpose involution: For any $n \geq 2$, take $x=1e_11-1e_22$, $y=1e_12+1e_21$ (where all the other components are zero). It is easily seen that $x,y \in S$.
$xy=1e_12-1e_21$ and $yx=-xy$, so $xy-yx=2xy=2e_12-2e_21 \neq 0$.
Denote any $s \in S$ by $s=ae_11+be_12+be_21+de_22+\ldots$ (one can see that the other components are not important in the computation of $xsy$, so we do not bother to write them).
A direct computation shows that $xsy=be_11+ae_12-de_21-be_22$. $xsy$ is indeed in $S^2$: $xsy=(be_11-be_22)+(ae_12-de_21)$.
The first term $be_11-be_22$ is obviously in $S$.
As for the second term: $ae_12-de_21=(ae_11-de_22)(1e_12+1e_21)$, obviously $ae_11-de_22 \in S$ and $1e_12+1e_21 \in S$, so $ae_12-de_21 \in S^2$.
Therefore, $xSy \subseteq S^2$ and Corollary \ref{xsy in S^2 and simple then S^2=R} can be applied.
\item The symplectic involution: For any $n=2m \geq 4$ (remember that if $n=2$, then $S \cong F$, so $S^2=S \cong F \subsetneq R$), take $x=1e_{1,m+2}-1e_{2,m+1}$, $y=1e_12+1e_{m+2,m+1}$ (where all the other components are zero).
It is easily seen that $x,y \in S$.
$xy=1e_{1,m+1}$ and $yx=-xy$, so $xy-yx=2xy=2e_{1,m+1} \neq 0$.
Denote any $s \in S$ by $s=(a_{ij})$.
A direct computation shows that $xsy=A+B+C$ where $A=a_{m+2,1}e_{12}$, $B=a_{22}e_{1,m+1}$, $C=-a_{21}e_{2,m+1}$. It is not difficult to see that $e_{12},e_{1,m+1},e_{2,m+1} \in S^2$.
Indeed,
\bit
\item $(e_{11}+e_{m+1,m+1})(e_{12}+e_{m+2,m+1})=
e_{11}e_{12}+e_{11}e_{m+2,m+1}+e_{m+1,m+1}e_{12}+e_{m+1,m+1}e_{m+2,m+1}=e_{12}$,
so $e_{12} \in S^2$ (of course, $e_{11}+e_{m+1,m+1} \in S$).
\item $xy=(e_{1,m+2}-e_{2,m+1})(e_{12}+e_{m+2,m+1})=e_{1,m+2}e_{12}+e_{1,m+2}+e_{m+2,m+1}-e_{2,m+1}e_{12}-e_{2,m+1}e_{m+2,m+1}=e_{1,m+1}$.
\item
$(e_{22}+e_{m+2,m+2})x=(e_{22}+e_{m+2,m+2})(e_{1,m+2}-e_{2,m+1})=e_{22}e_{1,m+2}-e_{22}e_{2,m+1}+e_{m+2,m+2}e_{1,m+2}-e_{m+2,m+2}e_{2,m+1}=-e_{2,m+1}$ (of course, $e_{22}+e_{m+2,m+2} \in S$).
\eit
Hence $xsy$ is indeed in $S^2$. Therefore, $xSy \subseteq S^2$ and Corollary \ref{xsy in S^2 and simple then S^2=R} can be applied.
\eit
\end{example}


Now it is time to bring the following criterion for $S^2=R$, which is just a combination of Corollary \ref{simple S^3=R} and of Corollary \ref{xsy in S^2 and simple then S^2=R}.

\begin{thm}[A first possible criterion for a positive answer to Herstein's question]\label{first criterion} 
Let $R$ be a simple associative unital $F$-algebra with $\Char(F)\neq 2$. Let $*$ be an involution on $R$ of the first kind and $S \subsetneq R$. Assume $S$ is not a commutative set.
Then: there exist $x,y \in S$ such that $xy-yx \neq 0$ and for every $s \in S$, $xsy \in S^2$ $\Leftrightarrow$ $S^2=R$.
\end{thm}

\begin{proof}
\bit
\item $\Rightarrow$: If there exist $x,y \in S$ such that $xy-yx \neq 0$ and for every $s \in S$, $xsy \in S^2$, then Corollary \ref{xsy in S^2 and simple then S^2=R} says that $S^2=R$.
\item $\Leftarrow$: Since $S$ is not a commutative set, there exist $x,y \in S$ such that $xy-yx \neq 0$. $S^2=R$ so every element $r \in R$ is in $S^2$, in particular, for every $s \in S$, $xsy \in R=S^2$.
Another way to see this direction: Otherwise, for every $x,y \in S$ such that $xy-yx \neq 0$, exists $s(x,y) \in S$ such that $xs(x,y)y \notin S^2$.
Now since $S$ is not a commutative set there do exist $x,y \in S$ such that $xy-yx \neq 0$. Therefore, for this pair $x,y$ exists $s(x,y) \in S$ such that $xs(x,y)y \notin S^2$.
But by definition of $S^3$, $xs(x,y)y \in S^3$, so we have $S^2 \subsetneq S^3$. Finally, Corollary \ref{simple S^3=R} says that $S^3=R$, so $S^2 \subsetneq S^3=R$.
\eit
\end{proof}

\begin{remark}[Possible criterion in practice]\label{remark possible criterion} 

Let $R$ be a simple associative unital $F$-algebra with $\Char(F)\neq 2$. Let $*$ be an involution on $R$ of the first kind and $S \subsetneq R$. Assume $S$ is not a commutative set.
The first possible criterion says that there exist $x,y \in S$ such that $xy-yx \neq 0$ and for every $s \in S$, $xsy \in S^2$ $\Leftrightarrow$ $S^2=R$.
However, given such an algebra, in order to show that $S^2 \neq R$, it is enough to find $x,y,s \in S$ such that $xsy \notin S^2$, which seems a little easier then to find for every noncommuting $x,y \in S$, an element (which depends on these $x,y$) $s=s(x,y)$ such that $xs(x,y)y \notin S^2$.
Indeed, otherwise for any triple $s_1,s_2,s_3 \in S$ $s_1s_2s_3 \in S^2$, hence $S^3 \subseteq S^2$, so (since $S^3 \supseteq S^2$) $S^3=S^2$. But Corollary \ref{simple S^3=R} says that $S^3=R$, therefore $S^2=R$.
\end{remark}

Notice that there is no restriction on $dim_{Z(R)}R$ (we can have $dim_{Z(R)}R=4$, as in $\M_2(F)$ with the transpose involution, $\Char(F)\neq 2$);
Again, if we insist that $dim_{Z(R)}R$ will enter our results, then we may have the following:

\begin{cor}\label{first criterion with dim>4}
Let $R$ be a simple associative unital $F$-algebra with $\Char(F)\neq 2$. Let $*$ be an involution on $R$ of the first kind such that $Z(R) \subsetneq S$ (and not just $Z(R) \subseteq S$) and $S \subsetneq R$. Assume that $dim_{Z(R)}R > 4$.
Then: there exist $x,y \in S$ such that $xy-yx \neq 0$ and for every $s \in S$, $xsy \in S^2$ $\Leftrightarrow$ $S^2=R$.
\end{cor}

\begin{proof}
\bit
\item $\Rightarrow$: If there exist $x,y \in S$ such that $xy-yx \neq 0$ and for every $s \in S$, $xsy \in S^2$, then Corollary \ref{xsy in S^2 and simple then S^2=R} says that $S^2=R$.
\item $\Leftarrow$: From Corollary \ref{cor of centS in Z} $S$ is not a commutative set. Now proceed exactly as in the proof of Theorem \ref{first criterion}.
\eit
\end{proof}

If one prefers to dismiss of the condition that $R$ is simple, but still have $S^2=R$, then we offer the following.
A word of caution: There is a difference between (ii) of Proposition \ref{section S^3: xy-yx is invertible} and (ii) of the following proposition, Proposition \ref{section S^2: xy-yx is invertible};
In (ii) of Proposition \ref{section S^3: xy-yx is invertible} we give an upper bound for the number of monomials of length $3$ in the representation of every element $r \in R$, namely: $5$ monomials at most are needed.
However, in (ii) of Proposition \ref{section S^2: xy-yx is invertible} we do not give an upper bound for the number of monomials of length $2$ in the representation of every element $r \in R$.
We only know that there exist $x,y \in S$ such that $xy-yx \neq 0$ and for every $s \in S$, $xsy \in S^2$.
Generally, $xsy \in S^2$ just tell us that $xsy=\sum_{1 \leq i \leq m(s)} a_ib_i$ where $a_i,b_i \in S$ and $m=m(s)$ can be any finite number (which depends on the particular $s \in S$).
If there exists $M \in \N$ such that for every $s \in S$, $m(s) \leq M$, then it is possible to give an upper bound for the number of monomials of length $2$ in the representation of every element $r \in R$, namely $M+5$ monomials at most are needed.

\begin{prop}\label{section S^2: xy-yx is invertible}
Let $R$ be an associative unital $F$-algebra with $\Char(F)\neq 2$. Let $*$ be an involution on $R$ of the first kind and $S \subsetneq R$. Assume $S$ is not a commutative set. Then, 
\bit
\item [(i)] Assume that there exist $x,y \in S$ such that $xy-yx \neq 0$, for every $s \in S$, $xsy \in S^2$, and $1$ can be written as $\sum_{1 \leq i \leq m}u_i(xy-yx)v_i$ where $u_i,v_i \in R$. Then $S^2=R$.
\item [(ii)] Assume that a stronger condition then the condition in (i) is satisfied, namely, that there exist $x,y \in S$ such that $xy-yx$ is right invertible and for every $s \in S$, $xsy \in S^2$.
Also assume that there exists $M \in \N$ such that for every $s \in S$, $xsy=\sum_{1 \leq i \leq M} a_i(s)b_i(s)$ where $a_i(s),b_i(s) \in S$.
\eit
Then (in addition to $S^2=R$) each element of $R$ can be written in the form $\lambda+ \mu$, 
where $\lambda \in S$, 
$\mu =a_1b_1+a_2b_2+\ldots+a_{M+3}b_{M+3}+a_{M+4}b_{M+4}$, with $a_1 ,b_1,\ldots, a_{M+4}, b_{M+4} \in S$.
\end{prop}

\begin{proof}
\bit
\item [(i)] Write $1=\sum_{1 \leq i \leq m}u_i(xy-yx)v_i$ where $u_i,v_i \in R$ and $x,y \in S$ such that $xy-yx \neq 0$ and for every $s \in S$, $xsy \in S^2$.
This of course implies that the non-zero two-sided ideal $R(xy-yx)R$ equals $R$. But from Lemma \ref{second lemma xsy in S^2 then r(xy-yx)t in S^2}, $R(xy-yx)R$ is contained in $S^2$. Therefore, $S^2=R$.
\item [(ii)] By assumption there exists $x,y \in S$ such that $xy-yx$ is right invertible, so there exists $z \in R$ such that $(xy-yx)z=1$.
For every $r\in R$, $rxy+xyr^*=rxy+yxr^*-yxr^*+xyr^*$. Let $\alpha=\alpha(r)=rxy+xyr^*$, $\beta=\beta(r)=rxy+yxr^*$.  So for every $r \in R$, $(xy-yx)r^*=\alpha(r)-\beta(r)$.
We have seen in the proof of Lemma \ref{first lemma xsy in S^2 then (xy-yx)r in S^2} that $\beta(r) \in S$ and $\alpha(r) \in S^2$.
More elaborately for $\alpha(r)$, we have seen that if $r=s+k$ with $s \in S$ and $k \in K$:
$\alpha(r)=rxy+xyr^*=\ldots=(sxy+xys)+[(kx-xk)y-x(yk-ky)]=(sx+xs)y+x(ys+sy)-2xsy+[(kx-xk)y-x(yk-ky)]=(sx+xs)y+x(ys+sy)-2[\sum_{1 \leq i \leq M}a_i(s)b_i(s)]+[(kx-xk)y-x(yk-ky)]$,
where we have used the extra assumption that there exists $M \in \N$ such that for every $s \in S$, $xsy=\sum_{1 \leq i \leq M} a_i(s)b_i(s)$ where $a_i(s),b_i(s) \in S$.
So denote: $\lambda(r)=-\beta(r)=-(rxy+yxr^*) \in S$, $\mu(r)=\alpha(r)=(sx+xs)y+x(ys+sy)-2[\sum_{1 \leq i \leq M}a_i(s)b_i(s)]+[(kx-xk)y-x(yk-ky)]$.
Hence, $(xy-yx)r^*=-\beta(r)+\alpha(r)=\lambda(r)+\mu(r)$.
Define $a_1(r)=kx-xk, b_1(r)=y, a_2(r)=-x, b_2(r)=yk-ky, a_3(r)=sx+xs, b_3(r)=y, a_4(r)=x, b_4(r)=ys+sy, a_5(r)=a_1(s), b_5(r)=b_1(s), a_6(r)=a_2(s), b_6(r)=b_2(s), \ldots, a_{M+4}(r)=a_M(s), b_{M+4}(r)=b_M(s)$.
Therefore, $(xy-yx)r^*=-\beta(r)+\alpha(r)=\lambda(r)+\mu(r)=\lambda(r)+a_1(r)b_1(r)+a_2(r)b_2(r)+\ldots+a_{M+4}(r)b_{M+4}(r)$.
Now, for $r \in R$, $r=1r=((xy-yx)z)r=(xy-yx)(zr)$ (here we use the right invertibility of $xy-yx$).

In our notations, 

$r=(xy-yx)(zr)=-\beta((zr)^*)+\alpha((zr)^*)=\lambda((zr)^*)+\mu((zr)^*)=
\lambda(zr)^*+a_1((zr)^*)b_1((zr)^*)+a_2((zr)^*)b_2((zr)^*)+\ldots+a_{M+4}((zr)^*)b_{M+4}((zr)^*)$.
So each element $r \in R$ can be written in the requested form.
\eit
\end{proof}

\begin{example}\label{section S^2 without simple: example matrices}
Let $R=\M_n(F)$, where $F$ is a field of characteristic not equal to $2$, $n>1$.
As for (ii) of Proposition \ref{section S^2: xy-yx is invertible} one must be careful:
\bit
\item The transpose involution: For $n=2$, if one takes $x$ and $y$ as in Example \ref{section S^2: example matrices}, namely $x=1e_11-1e_22$ and $y=1e_12+1e_21$, then $xy-yx$ is invertible, $xSy \subseteq S^2$ and $2$ is such that for every $s \in S$, $xsy=\sum_{1 \leq i \leq 2} a_i(s)b_i(s)$ where $a_i(s),b_i(s) \in S$.
Therefore, we can use (ii) of Proposition \ref{section S^2: xy-yx is invertible}.
However, for $n > 2$ for those $x$ and $y$, $xy-yx$ is not invertible.
\item The symplectic involution: For $n=2m \geq 4$, let $x$ and $y$ be as in Example \ref{section S^2: example matrices}, namely $x=1e_{1,m+2}-1e_{2,m+1}$, $y=1e_12+1e_{m+2,m+1}$. $xy-yx$ is not invertible.
\eit

This only shows that the particular $x$ and $y$ of the transpose involution and the particular $x$ and $y$ of the symplectic involution from Example \ref{section S^2: example matrices}, do not satisfy (ii) of Proposition \ref{section S^2: xy-yx is invertible}.

\end{example}

We would like to remark about the dimension of a simple $R$ over $Z(R)$.
Unfortunately, we are not able to give an analog theorem to Theorem \ref{simple with dim > 4}.
Namely, if $R$ is a simple associative unital $F$-algebra with $\Char(F)\neq 2$, $*$ is an involution on $R$ of the first kind such that $Z(R) \subsetneq S$ (and not just $Z(R) \subseteq S$), $S \subsetneq R$ and $dim_{Z(R)}R > 4$, then not necessarily that $S^2=R$ (remember that Theorem \ref{simple with dim > 4} says that $S^3=R$).
This is because Corollary \ref{cor of centS in Z} just says that $S$ is a noncommutative set (since $Z(R) \subsetneq S$), but it does not say that there exist $x,y \in S$ such that $xy-yx \neq 0$ and $xSy \subseteq S^2$.
In order to use Corollary \ref{xsy in S^2 and simple then S^2=R} we need to find $x,y \in S$ such that $xy-yx \neq 0$ and $xSy \subseteq S^2$.
If we insist that $dim_{Z(R)}R$ will enter our results, then we may have the following. However, it is just Corollary \ref{xsy in S^2 and simple then S^2=R} with "$Z(R) \subsetneq S$ (and not just $Z(R) \subseteq S$) and $dim_{Z(R)}R > 4$" instead of "$S$ is not a commutative set".

\begin{thm}\label{simple with dim and more}
Let $R$ be a simple associative unital $F$-algebra with $\Char(F)\neq 2$. Let $*$ be an involution on $R$ of the first kind such that $Z(R) \subsetneq S$ (and not just $Z(R) \subseteq S$) and $S \subsetneq R$. Assume that $dim_{Z(R)}R > 4$. Also assume that there exist $x,y \in S$ such that $xy-yx \neq 0$ and $xSy \subseteq S^2$.
Then $S^2=R$.
\end{thm}


We have seen in corollary \ref{xsy in S^2 and simple then S^2=R} that if $R$ is a simple associative unital $F$-algebra with $\Char(F)\neq 2$, $*$ is an involution on $R$ of the first kind, $S \subsetneq R$ and there exist $x,y \in S$ such that $xy-yx \neq 0$ and for every $s \in S$, $xsy \in S^2$, then $S^2=R$.
However, one can give other criteria for $S^2=R$. We shall bring now another possible criterion for $S^2$ to equal $R$.
Notice that we continue to assume that $S$ is not a commutative set. Also notice that in the following criterion the particular $x,y \in S$ satisfy $xy+yx \neq 0$.
We can follow the same path as in the first criterion, namely: starting with two lemmas (similar to Lemma \ref{first lemma xsy in S^2 then (xy-yx)r in S^2} and Lemma \ref{second lemma xsy in S^2 then r(xy-yx)t in S^2}), then bringing a theorem and its corollary (similar to Theorem \ref{xsy in S^2 then two sided in S^2} and Corollary \ref{xsy in S^2 and simple then S^2=R}) and finally bringing the criterion (similar to Theorem \ref{first criterion}).
However, for the sake of brevity, we bring all the results in one theorem, namely the second possible criterion.

\begin{thm}[A second possible criterion for a positive answer to Herstein's question]\label{second criterion} 
Let $R$ be a simple associative unital $F$-algebra with $\Char(F)\neq 2$. Let $*$ be an involution on $R$ of the first kind and $S \subsetneq R$. Assume $S$ is not a commutative set.
Then: there exist $x,y \in S$ such that $xy+yx \neq 0$ and for every $k \in K$, $xky \in S^2$ $\Leftrightarrow$ $S^2=R$.
\end{thm}

Observe that reading the proof reveals that it is enough to demand that $xT_0y \subseteq S^2$ which is a little weaker condition then $xKy \subseteq S^2$.
However, in Lemma \ref{first lemma xsy in S^2 then (xy-yx)r in S^2} we could not say that reading the proof reveals that it is enough to demand that $xTy \subseteq S^2$ (instead of $xSy \subseteq S^2$).

\begin{proof}
\bit
\item $\Rightarrow$: Assume that there exist $x,y \in S$ such that $xy+yx \neq 0$ and for every $k \in K$, $xky \in S^2$. We must show that $S^2=R$.
For every $r \in R$, $rxy-xyr^*=rxy+yxr^*-yxr^*-xyr^*$.
Let $\alpha=\alpha(r)=rxy-xyr^*$, $\beta=\beta(r)=rxy+yxr^*$. We will just write $\alpha$ and $\beta$, instead of $\alpha(r)$ and $\beta(r)$. So for every $r \in R$, $\alpha=\beta-(xy+yx)r^*$.
Claim 1: $\alpha \in S^2$, $\beta \in S$.
Indeed, obviously $\beta \in S$, since $(\beta)^*=(rxy+yxr^*)*=yxr^*+rxy=\beta$.
As for $\alpha$: $\alpha=rxy-xyr^*=(rx+xr^*)y-x(yr^*+ry)-xr^*y+xry=(rx+xr^*)y-x(yr^*+ry)-x(r^*-r)y$.
Clearly, $rx+xr^*, y, -x, yr^*+ry \in S$. $r^*-r \in K$, so $x(r^*-r)y \in S^2$ (just by our assumption that $xKy \subseteq S^2$). Therefore, $\alpha=rxy-xyr^*=(rx+xr^*)y-x(yr^*+ry)-x(r^*-r)y \in S^2$.
Using the claim we get: for every $r \in R$, $-(xy+yx)r^*=\alpha-\beta \in S^2$ (remember that $S \subseteq S^2$). But $R=R^*$ (since $r=(r^*)^*$), so for every $r \in R$, $-(xy+yx)r \in S^2$. Then, obviously, for every $r \in R$, $(xy+yx)r \in S^2$
(Of course, from the first place we could take $-r^*$ instead of $r$, and get, without using $R=R^*$, that $(xy+yx)r \in S^2$).
Therefore, $(xy+yx)R \subseteq S^2$.
Claim 2: For every $u \in R$ and $w \in S^2$, $wu-uw \in S^2$ (in other words, $S^2$ is a Lie ideal of $R$).
We have already seen this claim in the proof of Lemma \ref{second lemma xsy in S^2 then r(xy-yx)t in S^2}.
Next, for every $r \in R$, $(xy+yx)r \in S^2$, so from the claim just seen (with $w=(xy+yx)r$), for every $u \in R$ and for every $r \in R$, $((xy+yx)r)u-u((xy+yx)r) \in S^2$.
Use again $(xy+yx)R \subseteq S^2$ to get that (for every $u \in R$ and for every $r \in R$) $((xy+yx)r)u=(xy+yx)ru \in S^2$.
Therefore, for every $u \in R$ and for every $r \in R$, $-u((xy+yx)r) \in S^2$. So $R(xy+yx)R \subseteq S^2$.
By assumption, $xy+yx \neq 0$, so $R(xy+yx)R$ is a non-zero two-sided ideal of $R$ ($0 \neq xy+yx=1(xy+yx)1 \in R(xy+yx)R$).
$R$ is simple, so $R(xy+yx)R=R$. But we have just seen that $R(xy+yx)R \subseteq S^2$, hence $R=S^2$.
\item $\Leftarrow$: Assume that $S^2=R$. Let $x=y=1$. So we have found $x,y \in S$ such that $xy+yx=2 \neq 0$ and for every $k \in K$, $(k=1k1=)xky \in S^2$ ($S^2=R$ so every element $r \in R$ is in $S^2$, in particular, for every $k \in K$, $xky \in R=S^2$.
\eit
\end{proof}

Similarly to Example \ref{section S^2: example matrices}, we have here:

\begin{example}\label{example for second criterion}

Let $R=\M_n(F)$, where $F$ is a field of characteristic not equal to $2$ and $n>1$.
Theorem \ref{second criterion} can be applied for both the transpose involution and the symplectic involution on $\M_n(F)$, thus showing that in each of these two cases $S^2=R$. Indeed;\bit
\item The transpose involution: Let $n \geq 2$. Take $x=e_{11}$, $y=e_{12}+e_{21}$.
One sees that $x,y \in S$ and $xy+yx=e_{11}(e_{12}+e_{21})+(e_{12}+e_{21})e_{11}=e_{12}+e_{21}=y \neq 0$.
A direct computation shows that if $K \ni (a_{ij})=k=ae_{12}-ae{21}+\ldots$, then $xky=ae_{11}$.
Indeed, $x(ae_{12}-e{21}+\sum_{j\neq 2}a_{1j}+\sum_{j \neq 1}a_{2j}+\sum_{i>3}\sum{j}a_{ij})y=
(e_{11}(ae_{12}-e{21}+\sum_{j\neq 2}a_{1j}e_{1j}+\sum_{j \neq 1}a_{2j}e_{2j}+\sum_{i>3}\sum{j}a_{ij}e_{ij}))(e_{12}+e_{21})=
(ae_{11}e_{12}-e_{11}e_{21}+e_{11}\sum_{j\neq 2}a_{1j}e_{1j}+e_{11}\sum_{j \neq 1}a_{2j}e_{2j}+e_{11}\sum_{i>3}\sum{j}a_{ij}e_{ij})(e_{12}+e_{21})=
(ae_{12}+\sum_{j\neq 2}a_{1j}e_{11}e_{1j})(e_{12}+e_{21})=
(ae_{12}+\sum_{j \neq 2}a_{1j}e_{1j})(e_{12}+e_{21})=
ae_{12}e_{12}+ae_{12}e_{21}+\sum_{j \neq 2}a_{1j}e_{1j}e_{12}+\sum_{j \neq 2}a_{1j}e_{1j}e_{21}=
ae_{11}+a_{11}e_{11}e_{12}=ae_{11}+a_{11}e_{12}=ae_{11}+0e_{12}=ae_{11}$ (of course $a_{11}=0$ since $(a_{ij})=k \in K$).
Clearly, $ae_{11} \in S \subseteq S^2$, so $xKy \subseteq S^2$. Therefore, Theorem \ref{second criterion} can be applied.
\item The symplectic involution: Let $n=2m \geq 4$ (remember that when $n=2$, $S$ is commutative, actually $S \cong F$, so $S^2=S \subsetneq R$).
Take $x=e_{1,m+2}-e_{2,m+1}$, $y=e_{11}+e_{m+1,m+1}$. One sees that $x,y \in S$.

$xy=(e_{1,m+2}-e_{2,m+1})(e_{11}+e_{m+1,m+1})=
e_{1,m+2}e_{11}+e_{1,m+2}e_{m+1,m+1}-e_{2,m+1}e_{11}-e_{2,m+1}e_{m+1,m+1}=
-e_{2,m+1}e_{m+1,m+1}=-e_{2,m+1}$.

$yx=(e_{11}+e_{m+1,m+1})(e_{1,m+2}-e_{2,m+1})=e_{11}e_{1,m+2}=e_{1,m+2}$.

Hence, $xy+yx=-e_{2,m+1}+e_{1,m+2} \neq 0$.
A direct computation shows that if $K \ni (a_{ij})=k$, 
then $xky=a_{m+2,1}e_{11}-a_{m+1,1}e_{21}+a_{m+2,m+1}e_{1,m+1}-a_{m+1,m+1}e_{2,m+1}$.
Each of the four terms is in $S^2$:
\bit
\item $e_{11}$: We only show that for $n=4$ $e_{11} \in S^2$.
$e_{11}=(1/4)(A+B+2C-D)$, where $A=e_{11}-e_{22}+e_{33}-e_{44}$, $B=e_{11}+e_{22}$, $C=e_{11}+e_{44}$, $D=e_{33}+e_{44}$.
Indeed, $(1/4)[(e_{11}-e_{22}+e_{33}-e_{44})+(e_{11}+e_{22})+2(e_{11}+e_{44})-(e_{33}+e_{44})]=
(1/4)[e_{11}-e_{22}+e_{33}-e_{44}+e_{11}+e_{22}+2e_{11}+2e_{44}-e_{33}-e_{44}]=(1/4)[4e_{11}]=e_{11}$.
It is not difficult to see that $A,B,C,D \in S^2$:
\bit
\item $A \in S$.
\item $(e_{14}-e_{23})(e_{41}-e_{32})=e_{14}e_{41}-e_{14}e_{32}-e_{23}e_{41}+e_{23}e_{32}=e_{11}+e_{22}=B$.
\item $(e_{12}+e_{43})(e_{12}+e_{21}+e_{34}+e_{43})=e_{12}e_{12}+e_{12}e_{21}+e{12}e_{34}+e_{12}e_{43}+e_{43}e_{12}+e_{43}e_{21}+e_{43}e_{34}+e_{43}e_{43}=e_{11}+e_{44}=C$.
\item $(e_{41}-e_{32})(e_{14}-e_{23})=e_{41}e_{14}-e_{41}e_{23}-e_{32}e_{14}+e_{32}e_{23}=e_{44}+e_{33}=D$.
\eit
\item (For any $n=2m \leq 4$): $(e_{22}+e_{m+2,m+2})(e_{21}+e_{m+1,m+2})=e_{22}e_{21}+e_{22}e_{m+1,m+2}+e_{m+2,m+2}e_{21}+e_{m+2,m+2}e_{m+1,m+2}=
e_{22}e_{21}=e_{21}$.
\item (For any $n=2m \leq 4$): $e_{1,m+1}, e_{2,m+1}$: we have already seen in Example \ref{section S^2: example matrices} that $e_{1,m+1},e_{2,m+1} \in S^2$. Just for convenience we show this again:
$(e_{1,m+2}-e_{2,m+1})(e_{12}+e_{m+2,m+1})=e_{1,m+2}e_{12}+e_{1,m+2}+e_{m+2,m+1}-e_{2,m+1}e_{12}-e_{2,m+1}e_{m+2,m+1}=e_{1,m+1}$.
 $(e_{22}+e_{m+2,m+2})(e_{1,m+2}-e_{2,m+1})=e_{22}e_{1,m+2}-e_{22}e_{2,m+1}+e_{m+2,m+2}e_{1,m+2}-e_{m+2,m+2}e_{2,m+1}=-e_{2,m+1}$.
\eit
So $xKy \subseteq S^2$. Therefore, Theorem \ref{second criterion} can be applied.
\eit
\end{example}

We do not bother to bring a result when $R$ is not necessarily a simple ring, but there exist $x,y \in S$ such that $xy+yx$ is right invertible and for every $k \in K$, $xky \in S^2$ (similar to Proposition \ref{section S^2: xy-yx is invertible}) or to bring a result when $R$ is simple and $dim_{Z(R)}R > 4$ (similar to Corollary \ref{first criterion with dim>4}).
It is not difficult to complete the details and have such results.

\section{Other results}

One can continue in the spirit of the theorems seen thus far and get further results concerning $S$ and $K$.
$R$ will continue to denote an associative $F$-algebra with $1$ and with an involution $*$ such that:
\bit
\item $F$ is a field of characteristic not equal to $2$ (hence $R=S+K$).
\item The involution $*$ is of the first kind, namely, $Z(R) \subseteq S$ (hence $S \subseteq S^2 \subseteq S^3$ etc.).
\item $S \subsetneq R$.
(Remember that if $S=R$, then $S^3 \supseteq S^2 \supseteq S=R$ so $S^3=R$ and $S^2=R$. We on purpose wrote separately $S^3=R$ and $S^2=R$ in order to emphasize what happens for $S^3$ and what happens for $S^2$).
\eit
However, in this section we shall consider subsets of $R$ other then $S^3$ and $S^2$:
\bit
\item [(i)] In the first subsection we shall mainly consider $K^6$, $K^4$, $K+KSK$ and $KS+K^2$.
The results about $K^6$ and $K^4$ follow immediately from results of Herstein and from our results.
For $K+KSK$ to equal $R$, we assume that $K$ is not a skew-commutative set.
For $KS+K^2$ (or $SK+K^2$) to equal $R$, we assume that there exist $x,y \in K$ such that $xy+yx \neq 0$ and $xSy \subseteq K^2$ (so in particular we assume that $K$ is not a skew-commutative set).
Notice the similarity to the second criterion which says that if there exist $x,y \in S$ such that $xy+yx \neq 0$ and $xKy \subseteq S^2$, then $S^2=R$.
Another option for $KS+K^2$ to equal $R$ is when there exist $x \in S$ and $y \in K$ such that $xy-yx \neq 0$ and $xSy \subseteq K^2$.
We also consider $K+K^2$, $K+K^2+K^3$ and $K+K^3$.
\item [(ii)] In the second subsection, in analogy to $S^3$, we shall consider $SKS$ and $K+KS+SK \subseteq SKS$.
Notice that if $S=R$, then (from $S \cap K =0$) $K=0$, hence $SKS=0 \subsetneq R$ (in contrast to $S^3=R$), so there is nothing interesting to say about $SKS$ or $K+KS+SK$ (both are zero).
\item [(iii)] In the third subsection, again in analogy to $S^3$, we shall consider $S^2K$ (and $SK \subseteq S^2K$) and $KS^2$ (and $KS \subseteq KS^2$).
Notice that if $S=R$ then $K=0$, hence $S^2K=0 \subsetneq R$ and $KS^2=0 \subsetneq R$ (in contrast to $S^3=R$), so there is nothing interesting to say.
\eit
\subsection{Results mostly about \texorpdfstring{$K^6$, $K^4$, $K+KSK$, $KS+K^2$}{K6, K4, K+KSK, KS+K2}}
\subsubsection{An almost analog theorem for $K^6$ and $K^4$}

We give a result dealing exclusively with $K$ instead of $S$, which follows immediately from two theorems of Herstein.
Of course, when we have dealt with $S$ we had $S \subseteq S^2\subseteq S^3$, however for $K$ we usually do not have such a chain of inclusions (Notice that in $\M_2(F)$ with the symplectic involution, $K \subsetneq K^2=\M_2(F)$, so we do have a chain of inclusions).
A combination of Theorem 2.2 and Theorem 2.3 of Herstein (see \cite{her1}) implies the next theorem.
Observe that in Herstein's Theorem 2.2 and Theorem 2.3:
\bit
\item There is no demand for $S$ to be a noncommutative set.
\item $S \subsetneq R$ is an implicit demand, since if $S=R$, then from $S \cap K =0$, one gets $K=0$, so $\bar{K}=0 \subsetneq R$.
\eit

\begin{thm}[Herstein]\label{herstein} 
Let $R$ be a simple associative unital $F$-algebra 

with $\Char(F)\neq 2$ and $dim_{Z(R)}R > 4$. 
Let $*$ be an involution on $R$ of the first kind and $S \subsetneq R$. Then, $R=K+(K\circ K)$.
\end{thm}

If one wishes to express $R$ as $K^n$ for some $n \geq 2$ instead of $K+(K\circ K)$ ($K+(K\circ K)$ is not contained in $K^2$, only $K \circ K$ is contained in $K^2$), then one may express $R$ as follows:

\begin{thm}\label{from Herstein 2.2 and 2.3}
Let $R$ be a simple associative unital $F$-algebra with $\Char(F)\neq 2$. Let $*$ be an involution on $R$ of the first kind and $S \subsetneq R$. Assume $S$ is not a commutative set. Then,
\bit
\item $K^6=R$, or more specifically $(K\circ K)^3=R$.
\item If, in addition, exist $x,y \in S$ such that $xy-yx \neq 0$ and for every $s \in S$, $xsy \in S^2$, then $K^4=R$, or more specifically $(K \circ K)^2=R$.
(If, in addition, exist $x,y \in S$ such that $xy+yx \neq 0$ and for every $k \in K$, $xky \in S^2$, then $K^4=R$, or more specifically $(K \circ K)^2=R$).
\eit
\end{thm}

\begin{proof}
Use Theorem 2.2 and Theorem 2.3 of Herstein (see \cite{her1}) (or the above Theorem \ref{herstein}) and get that in our $R$, $S=K \circ K$.
\bit
\item From Corollary \ref{simple S^3=R} $S^3=R$, so $K^6=(K^2)^3 \supseteq (K \circ K)^3=S^3=R$.
\item From Corollary \ref{xsy in S^2 and simple then S^2=R} $S^2=R$, so $K^4=(K^2)^2 \supseteq (K \circ K)^2=S^2=R$.
(From Theorem \ref{second criterion} $S^2=R$, so $K^4=(K^2)^2 \supseteq (K \circ K)^2=S^2=R$).
\eit
\end{proof}

Using Corollary \ref{cor of centS in Z}, Theorem \ref{from Herstein 2.2 and 2.3} can be written as:

\begin{thm}
Let $R$ be a simple associative unital $F$-algebra with $\Char(F)\neq 2$. Let $*$ be an involution on $R$ of the first kind with $Z(R) \subsetneq S$ (not just $Z(R)\subseteq S$) and $S \subsetneq R$. Assume that $dim_{Z(R)}R > 4$.
Then, \bit
\item $K^6=R$, or more specifically $(K\circ K)^3=R$.
\item If, in addition, exist $x,y \in S$ such that $xy-yx \neq 0$ and for every $s \in S$, $xsy \in S^2$, then $K^4=R$, or more specifically $(K \circ K)^2=R$.
(If, in addition, exist $x,y \in S$ such that $xy+yx \neq 0$ and for every $k \in K$, $xky \in S^2$, then $K^4=R$, or more specifically $(K \circ K)^2=R$).
\eit
\end{thm}

\begin{proof}
\bit
\item Use Theorem \ref{simple with dim > 4}.
\item Use Theorem \ref{simple with dim and more}.
(Clear).
\eit
\end{proof}

\begin{remark} [$K^2=R$ and $dim_{Z(R)}R=4$]
In Theorem \ref{from Herstein 2.2 and 2.3} we have assumed that $dim_{Z(R)}R > 4$ and got $K^6=R$ or $K^4=R$.
For $dim_{Z(R)}R=4$, already $K^2$ may equal $R$:
\bit
\item $R=\H=\R+\R i+\R j+\R k$. $S=\R$, $K=\R i+\R j+\R k$. $K^2=\H$, because:
\bit
\item $S \subseteq K^2$: $-1=ii \in K^2$, so for every $\epsilon \in S=\R$, $\epsilon=(-\epsilon)(-1) \in K^2$ ($K^2$ is a vector space over $\R$).
\item $K \subseteq K^2$: $i=jk \in K^2$, $j=ki \in K^2$, $k=ij \in K^2$, so since $K^2$ is a vector space over $\R$, $K=\R i+\R j+\R k \in K^2$.
\eit
So, $\H= S+K \subseteq K^2$, hence $\H =K^2$.
\item $R=\M_2(F)$ with the symplectic involution. $S=\{a(e_{11}+e_{22}) | a \in F \}$,
$K=\{a(e_{11}-e_{22})+be_{12}+ce_{21} | a,b,c \in F \}$ (in other words, the matrices with zero traces).
A computation shows that $K^2=R$; Indeed,
\bit
\item $e_{12}e_{21}=e_{11}$. $e_{12},e_{21} \in K$, so $e_{11} \in K^2$.
\item $(e_{11}-e_{22})e_{12}=e_{11}e_{12}-e_{22}e_{12}=e_{12}$. $(e_{11}-e_{22}), e_{12} \in K$, so $e_{12} \in K^2$.
\item $(e_{11}-e_{22})(-e_{21})=-e_{11}e_{21}+e_{22}e_{21}=e_{21}$. $(e_{11}-e_{22}), -e_{21} \in K$, so $e_{12} \in K^2$.
\item $e_{21}e_{12}=e_{22}$. $e_{21}, e_{12} \in K$, so $e_{22} \in K^2$.
\eit
\eit
\end{remark}


\subsubsection{Results about $K+KSK$}

Herstein's theorem says that if $R$ is a simple associative unital $F$-algebra with $\Char(F)\neq 2$, $dim_{Z(R)}R > 4$, $*$ is an involution on $R$ of the first kind and $S \subsetneq R$, then $R=K+(K\circ K)$ (hence $R=K+K^2$, since $R=K+(K\circ K) \subseteq K+K^2$).
We shall bring a weaker version of Herstein's theorem, namely Theorem \ref{R simple xy+yx neq 0 then R=K+KSK}.
In Theorem \ref{R simple xy+yx neq 0 then R=K+KSK} we will only get that $R=K+KSK$, so our result is weaker then Herstein's in the sense that we show that $R$ equals $K+KSK$ which contains $K+K^2 \supseteq K+(K \circ K)$.
Notice that we will not assume that $S$ is not a commutative set.
Also notice that in Theorem \ref{R simple xy+yx neq 0 then R=K+KSK} we will not assume that $dim_{Z(R)}R > 4$, but we will assume that there exist $x,y \in K$ such that $xy+yx \neq 0$, namely, $K$ is not a skew-commutative set.


\begin{remark}
Actually, if $dim_{Z(R)}R > 4$ (and, as usual, $\Char(F) \neq 2$) then $K$ is necessarily not a skew-commutative set.
This result is immediate from Herstein's theorem \ref{herstein}:
Assume that $K$ is skew-commutative. Then for every $a,b \in K$, $(ab)^*=(b^*)(a^*)=(-b)(-a)=ba=-ab$, so $(ab+ba)^*=-ab-ba=-(ab+ba)$. Therefore, (merely by the definition of $K \circ K$ as the $F$-subspace of $R$ generated by $ab+ba$ where $a,b \in K$), $(K \circ K) \subseteq K$.
Now, from Herstein's theorem $R=K+(K\circ K) \subseteq K+K=K$, so $R=K$. But this is impossible, since on the one hand $1 \in R=K$, so $1^*=-1$. On the other hand $1 \in F \subseteq Z(R) \subset S$, so $1^*=1$. Combining the two yields $-1=1$ hence $1+1=0$, but by assumption $\Char(F) \neq 2$.
However, here we wish that our theorems will not rely on Herstein's, so we will assume (also when $dim_{Z(R)}R > 4$) that $K$ is not a skew-commutative set.
Without Herstein's theorem we are only able to show the following:
Let $R$ be an associative $F$-algebra with $\Char(F) \neq 2$. Let $*$ be an involution on $R$ (and $S \subsetneq R$). If $K$ is a skew-commutative set, then for every $k \in K$, $k^2=0$.
Indeed, take $k \in K$. $K$ is skew-commutative, hence, in particular, $kk=-kk$. Therefore $2kk=kk+kk=0$, so $kk=0$ (by assumption, $\Char(F) \neq 2$).
Therefore, if there exists $k \in K$ such that $k^2 \neq 0$, then $K$ is not a skew-commutative set.
This assumption seems reasonable, at least in some special cases, for example if $R$ is also a domain.
Indeed, if $R$ is also a domain, then necessarily exists $k \in K$ such that $k^2 \neq 0$; Otherwise, for every $k \in K$, $k^2=0$. But each $k \neq 0$ is regular, so $K=0$. This implies $S=R$, a contradiction to our usual assumption that $S \subsetneq R$.
\end{remark}


\begin{thm}\label{R simple xy+yx neq 0 then R=K+KSK}
Let $R$ be a simple associative unital $F$-algebra with $\Char(F)\neq 2$. Let $*$ be an involution on $R$ of the first kind and $S \subsetneq R$. Assume that $K$ is not a skew-commutative set. Then $R=K+KSK$.
\end{thm}

\begin{proof}
$K$ is not a skew-commutative set, hence there exist $x,y \in K$ such that $xy \neq -yx$. Hence $xy+yx \neq 0$.
For every $r \in R$, $rxy+xyr^*=rxy-yxr^*+yxr^*+xyr^*=rxy-yxr^*+(xy+yx)r^*$.
Let $\alpha=\alpha(r)=rxy+xyr^*$, $\beta=\beta(r)=rxy-yxr^*$. We will just write $\alpha$ and $\beta$, instead of $\alpha(r)$ and $\beta(r)$. So for every $r \in R$, $\alpha=\beta+(xy+yx)r^*$.
Claim 1: $\alpha \in KSK$, $\beta \in K$.
Indeed, obviously $\beta \in K$, since $(\beta)^*=(rxy-yxr^*)*=yxr^*-rxy=-(rxy-yxr^*)=-\beta$.
As for $\alpha$: $\alpha=rxy+xyr^*=(rx+xr^*)y+x(yr^*+ry)-xr^*y-xry=(rx+xr^*)y+x(yr^*+ry)-x(r^*+r)y$. Clearly, $(rx+xr^*), y, x, (yr^*+ry) \in K$ and $-x(r^*+r)y \in KSK$. Therefore, $\alpha=rxy+xyr^*=\ldots=(rx+xr^*)y+x(yr^*+ry)-x(r^*+r)y \in KSK$ (notice that $K^2 \subseteq KSK$).

Using the claim we get: for every $r \in R$, $(xy+yx)r^*=-\beta+\alpha \in K+KSK$. But $R=R^*$, so for every $r \in R$, $(xy+yx)r \in K+KSK$.
(Of course, from the first place we could take $r^*$ instead of $r$, and get, without using $R=R^*$, that $(xy+yx)r \in K+KSK$).
Therefore, $(xy+yx)R \subseteq K+KSK$.
Claim 2: For every $u \in R$ and $w \in K+KSK$, $wu+u^*w \in K+KSK$.
Clearly, it is enough to show that for every $u \in R$, $a \in K$ and $b \in KSK$: $au+u^*a \in K$ and $bu+u^*b \in KSK$ (since if $w=a+b$ with $a \in K$ and $b \in KSK$, $wu+u^*w=(a+b)u+u^*(a+b)=au+bu+u^*a+u^*b=(au+u^*a)+(bu+u^*b) \in K+KSK$).
$(au+u^*a)^*=u^*a^*+a^*u=u^*(-a)+(-a)u=-(au+u^*a)$, so $au+u^*a \in K$.
As for $b \in KSK$, it is enough to consider $b=xyz$ where $x,z \in K$ and $y \in S$.
$bu+u^*b=(xyz)u+u^*(xyz)=xy(zu+u^*z)+(u^*x+xu)yz-xyu^*z-xuyz=xy(zu+u^*z)+(u^*x+xu)yz-x(yu^*+uy)z$.
$(zu+u^*z),(u^*x+xu) \in K$ and $(yu^*+uy) \in S$, so $bu+u^*b=(xyz)u+u^*(xyz)=\ldots=xy(zu+u^*z)+(u^*x+xu)yz-x(yu^*+uy)z \in KSK$.
Next, for every $r \in R$, $(xy+yx)r \in K+KSK$, so from the claim just seen (with $w=(xy+yx)r$), for every $u \in R$ and for every $r \in R$, $((xy+yx)r)u+u^*((xy+yx)r) \in K+KSK$.
Use again $(xy+yx)R \subseteq K+KSK$ to get that (for every $u \in R$ and for every $r \in R$) $((xy+yx)r)u=(xy+yx)ru \in K+KSK$.
Therefore, for every $u \in R$ and for every $r \in R$, $u^*((xy+yx)r) \in K+KSK$. So $R(xy+yx)R \subseteq K+KSK$.
By assumption, $xy+yx \neq 0$, so $R(xy+yx)R$ is a non-zero two-sided ideal of $R$ ($0 \neq xy+yx=1(xy+yx)1 \in R(xy+yx)R$).
$R$ is simple, so $R(xy+yx)R=R$. But we have just seen that $R(xy+yx)R \subseteq K+KSK$, hence $R=K+KSK$.
\end{proof}


\subsubsection{Results about $KS+K^2$}

\begin{thm}\label{R simple two options for R=KS+K^2}
Let $R$ be a simple associative unital $F$-algebra with $\Char(F)\neq 2$. Let $*$ be an involution on $R$ of the first kind and $S \subsetneq R$. Then, \bit
\item [(1)] If there exist $x,y \in K$ such that $xy+yx \neq 0$ and $xSy \subseteq KS+K^2$, then $KS+K^2=R$ (If there exist $x,y \in K$ such that $xy+yx \neq 0$ and $xSy \subseteq SK+K^2$, then $SK+K^2=R$).
\item [(2)] If there exist $x \in K$ and $y \in S$ such that $xy-yx \neq 0$ and $xSy \subseteq KS+K^2$, then $KS+K^2=R$ (If there exist $x \in S$ and $y \in K$ such that $xy-yx \neq 0$ and $xSy \subseteq SK+K^2$, then $SK+K^2=R$).
\item [(3)] If there exist $x \in K$ and $y \in S$ such that $xy-yx \neq 0$ and $xSy \subseteq KS$, then $KS=R$ (If there exist $x \in S$ and $y \in K$ such that $xy-yx \neq 0$ and $xSy \subseteq SK$, then $SK=R$).
\eit
\end{thm}

\begin{proof}
\bit
\item [(1)] Let $x,y \in K$ such that $xy+yx \neq 0$ and $xSy \subseteq KS+K^2$.
For every $r \in R$, $rxy+xyr^*=rxy-yxr^*+yxr^*+xyr^*=rxy-yxr^*+(xy+yx)r^*$.
Let $\alpha=\alpha(r)=rxy+xyr^*$, $\beta=\beta(r)=rxy-yxr^*$. We will just write $\alpha$ and $\beta$, instead of $\alpha(r)$ and $\beta(r)$. So for every $r \in R$, $\alpha=\beta+(xy+yx)r^*$.
Claim 1: $\alpha \in KS+K^2$, $\beta \in K$.
Indeed, obviously $\beta \in K$, since $(\beta)^*=(rxy-yxr^*)*=yxr^*-rxy=-(rxy-yxr^*)=-\beta$.
As for $\alpha$: $\alpha=rxy+xyr^*=(rx+xr^*)y+x(yr^*+ry)-xr^*y-xry=(rx+xr^*)y+x(yr^*+ry)-x(r^*+r)y$.
Clearly, $(rx+xr^*), y, x, (yr^*+ry) \in K$.
$-x(r^*+r)y \in K^2$ from our assumption that $xSy \subseteq KS+K^2$.
Therefore, $\alpha=rxy+xyr^*=\ldots=(rx+xr^*)y+x(yr^*+ry)-x(r^*+r)y \in KS+K^2$.
Using the claim we get: for every $r \in R$, $(xy+yx)r^*=-\beta+\alpha \in KS+K^2$. But $R=R^*$, so for every $r \in R$, $(xy+yx)r \in KS+K^2$.
(Of course, from the first place we could take $r^*$ instead of $r$, and get, without using $R=R^*$, that $(xy+yx)r \in KS+K^2$).
Therefore, $(xy+yx)R \subseteq KS+K^2$.

Claim 2: For every $u \in R$ and $w \in KS+K^2$, $wu-uw \in KS+K^2$.
Clearly, it is enough to show that: \bit
\item For every $u \in R$, $a \in K$ and $b \in S$, $abu-uab \in KS$.
\item For every $u \in R$, $a,b \in K$, $abu-uab \in K^2$.
\eit
(since if $KS+K^2 \ni w=w_1+w_2$ with $w_1 \in KS$ and $w_2 \in K^2$, then $wu-uw=(w_1+w_2)u-u(w_1+w_2)=(w_1u-uw_1)+(w_2u-uw_2) \in KS+K^2$).
Indeed; \bit
\item For every $u \in R$, $a \in K$ and $b \in S$, $abu-uab=a(bu+u^*b)-(ua+au^*)b \in KS+KS=KS$ (obviously $(bu+u^*b) \in S$ and $(ua+au^*) \in K$).
\item For every $u \in R$, $a,b \in K$, $abu-uab=a(bu+u^*b)-(ua+au^*)b \in K^2$.
\eit
Next, for every $r \in R$, $(xy+yx)r \in KS+K^2$, so from the claim just seen (with $w=(xy+yx)r$), for every $u \in R$ and for every $r \in R$, $((xy+yx)r)u-u((xy+yx)r) \in KS+K^2$.
Use again $(xy+yx)R \subseteq KS+K^2$ to get that (for every $u \in R$ and for every $r \in R$) $((xy+yx)r)u=(xy+yx)ru \in KS+K^2$.
Therefore, for every $u \in R$ and for every $r \in R$, $-u((xy+yx)r) \in KS+K^2$. So $R(xy+yx)R \subseteq KS+K^2$.
By assumption, $xy+yx \neq 0$, so $R(xy+yx)R$ is a non-zero two-sided ideal of $R$.
$R$ is simple, so $R(xy+yx)R=R$. But we have just seen that $R(xy+yx)R \subseteq KS+K^2$, hence $R=KS+K^2$.
Similarly, one can show that $SK+K^2=R$ (the changes that must be made are: $(xy+yx)R \subseteq K+K^2 \subseteq SK+K^2$ and for every $u \in R$ and $w \in SK+K^2$, $wu-uw \in SK+K^2$).
\item [(2)] Let $x \in K$ and $y \in S$ such that $xy-yx \neq 0$ and $xSy \subseteq KS+K^2$.
For every $r \in R$, $rxy+xyr^*=rxy+yxr^*-yxr^*+xyr^*=rxy+yxr^*+(xy-yx)r^*$.
Let $\alpha=\alpha(r)=rxy+xyr^*$, $\beta=\beta(r)=rxy+yxr^*$.
Claim 1: $\alpha \in KS+K^2$, $\beta \in K$.
Indeed, obviously $\beta \in K$, since $(\beta)^*=(rxy+yxr^*)*=-yxr^*-rxy=-(rxy+yxr^*)=-\beta$.
As for $\alpha$: $\alpha=rxy+xyr^*=(rx+xr^*)y+x(yr^*+ry)-xr^*y-xry=(rx+xr^*)y+x(yr^*+ry)-x(r^*+r)y$.
$(rx+xr^*), x \in K$, $y, (yr^*+ry) \in S$ and $-x(r^*+r)y \in KS+K^2$ (by assumption $xSy \subseteq KS+K^2$).
Therefore, $\alpha=rxy+xyr^*=\ldots=(rx+xr^*)y+x(yr^*+ry)-x(r^*+r)y \in KS+K^2$.
Using the claim we get: for every $r \in R$, $(xy-yx)r^*=-\beta+\alpha \in K+KS+K^2 \subseteq KS+K^2$ (of course $K \subseteq KS$). But $R=R^*$, so for every $r \in R$, $(xy-yx)r \in KS+K^2$.
(Of course, from the first place we could take $r^*$ instead of $r$, and get, without using $R=R^*$, that $(xy-yx)r \in KS+K^2$).
Therefore, $(xy-yx)R \subseteq KS+K^2$.

Claim 2: For every $u \in R$ and $w \in KS+K^2$, $wu-uw \in KS+K^2$.
Clearly, it is enough to show that for every $u \in R$, $a,c \in K$ and $b \in S$, $abu-uab \in KS$ and $acu-uac \in K^2$.
Indeed; For every $u \in R$, $a,c \in K$ and $b \in S$, $abu-uab=a(bu+u^*b)-(ua+au^*)b \in KS+KS=KS$ (obviously $(bu+u^*b) \in S$ and $(ua+au^*) \in K$) and $acu-uac=a(cu+u^*c)-(ua+au^*)c \in K^2$ (obviously $(cu+u^*c),(ua+au^*) \in K$).
Next, for every $r \in R$, $(xy-yx)r \in KS+K^2$, so from the claim just seen (with $w=(xy-yx)r$), for every $u \in R$ and for every $r \in R$, $((xy-yx)r)u-u((xy-yx)r) \in KS+K^2$.
Use again $(xy-yx)R \subseteq KS+K^2$ to get that (for every $u \in R$ and for every $r \in R$) $((xy-yx)r)u=(xy-yx)ru \in KS+K^2$.
Therefore, for every $u \in R$ and for every $r \in R$, $-u((xy-yx)r) \in KS+K^2$. So $R(xy+yx)R \subseteq KS+K^2$.
By assumption, $xy-yx \neq 0$, so $R(xy-yx)R$ is a non-zero two-sided ideal of $R$.
$R$ is simple, so $R(xy-yx)R=R$. But we have just seen that $R(xy-yx)R \subseteq KS+K^2$, hence $R=KS+K^2$.
Similarly, one can show that $SK+K^2=R$ (the changes that must be made are: $(xy-yx)R \subseteq SK+K^2$ and for every $u \in R$ and $w \in SK+K^2$, $wu-uw \in SK+K^2$).
\item [(3)] Let $x \in K$ and $y \in S$ such that $xy-yx \neq 0$ and $xSy \subseteq KS$.
For every $r \in R$, $rxy+xyr^*=rxy+yxr^*-yxr^*+xyr^*=rxy+yxr^*+(xy-yx)r^*$.
Let $\alpha=\alpha(r)=rxy+xyr^*$, $\beta=\beta(r)=rxy+yxr^*$.
Claim 1: $\alpha \in KS$, $\beta \in K$.
Indeed, obviously $\beta \in K$, since $(\beta)^*=(rxy+yxr^*)*=-yxr^*-rxy=-(rxy+yxr^*)=-\beta$.
As for $\alpha$: $\alpha=rxy+xyr^*=(rx+xr^*)y+x(yr^*+ry)-xr^*y-xry=(rx+xr^*)y+x(yr^*+ry)-x(r^*+r)y$.
$(rx+xr^*), x \in K$, $y, (yr^*+ry) \in S$ and $-x(r^*+r)y \in KS$ (by assumption $xSy \subseteq KS$).
Therefore, $\alpha=rxy+xyr^*=\ldots=(rx+xr^*)y+x(yr^*+ry)-x(r^*+r)y \in KS$.
Using the claim we get: for every $r \in R$, $(xy-yx)r^*=-\beta+\alpha \in K+KS \subseteq KS$ (of course $K \subseteq KS$). But $R=R^*$, so for every $r \in R$, $(xy-yx)r \in KS$.
(Of course, from the first place we could take $r^*$ instead of $r$, and get, without using $R=R^*$, that $(xy-yx)r \in KS$).
Therefore, $(xy-yx)R \subseteq KS$.
Claim 2: For every $u \in R$ and $w \in KS$, $wu-uw \in KS$.
Clearly, it is enough to show that for every $u \in R$, $a \in K$ and $b \in S$, $abu-uab \in KS$.
Indeed; For every $u \in R$, $a \in K$ and $b \in S$, $abu-uab=a(bu+u^*b)-(ua+au^*)b \in KS+KS=KS$ (obviously $(bu+u^*b) \in S$ and $(ua+au^*) \in K$).

Next, for every $r \in R$, $(xy-yx)r \in KS$, so from the claim just seen (with $w=(xy-yx)r$), for every $u \in R$ and for every $r \in R$, $((xy-yx)r)u-u((xy-yx)r) \in KS$.
Use again $(xy-yx)R \subseteq KS$ to get that (for every $u \in R$ and for every $r \in R$) $((xy-yx)r)u=(xy-yx)ru \in KS$.
Therefore, for every $u \in R$ and for every $r \in R$, $-u((xy-yx)r) \in KS$. So $R(xy+yx)R \subseteq KS$.
By assumption, $xy-yx \neq 0$, so $R(xy-yx)R$ is a non-zero two-sided ideal of $R$.
$R$ is simple, so $R(xy-yx)R=R$. But we have just seen that $R(xy-yx)R \subseteq KS$, hence $R=KS$.
Similarly, one can show that $SK=R$ (the changes that must be made are: $(xy-yx)R \subseteq SK$ and for every $u \in R$ and $w \in SK+K^2$, $wu-uw \in SK+K^2$).
\eit
\end{proof}

\begin{example}
Matrices $M_n(F)$ ($\Char(F) \neq 2$) satisfy the first condition and the second condition of the above theorem.
\bit
\item [(1)] The first condition says: If there exist $x,y \in K$ such that $xy+yx \neq 0$ and $xSy \subseteq KS+K^2$, then $KS+K^2=R$.
Matrices (with $n >2$) satisfy a stronger condition, namely: there exist $x,y \in K$ such that $xy+yx \neq 0$ and $xSy \subseteq K^2$.
\bit
\item The transpose involution:
$\M_2(F)=KS+K^2$ (and $\M_2(F)=SK+K^2$) can be easily seen as follows, without using Theorem \ref{R simple two options for R=KS+K^2}
(The following computation seems better then using the first condition):
\bit
\item $KS \ni (e_{12}-e_{21})e_{11}=e_{12}e_{11}-e_{21}e_{11}=-e_{21}e_{11}=-e_{21}$.
\item $KS \ni (e_{12}-e_{21})e_{22}=e_{12}e_{22}-e_{21}e_{22}=e_{12}e_{22}=e_{12}$.
\item $KS \ni A=(e_{12}-e_{21})(e_{12}+e_{21})=e_{12}e_{12}+e_{12}e_{21}-e_{21}e_{12}-e_{21}e_{21}=e_{12}e_{21}-e_{21}e_{12}=e_{11}-e_{22}$.
$K^2 \ni B=(e_{12}-e_{21})(e_{12}-e_{21})=e_{12}e_{12}-e_{12}e_{21}-e_{21}e_{12}+e_{21}e_{21}=-e_{12}e_{21}-e_{21}e_{12}=-e_{11}-e_{22}$.
$KS+K^2 \ni A+B=(e_{11}-e_{22})+(-e_{11}-e_{22})=-2e_{22}$.
$KS+K^2 \ni A-B=(e_{11}-e_{22})-(-e_{11}-e_{22})=2e_{11}$.
\eit
So, $e_{11},e_{12},e_{21},e_{22} \in KS+K^2$, hence $\M_2(F)=KS+K^2$.

For $n \geq 3$: Take $x=e_{12}-e_{21} \in K$, $y=e_{13}-e_{31} \in K$.
$xy-yx=(e_{12}-e_{21})(e_{13}-e_{31})+(e_{13}-e_{31})(e_{12}-e_{21})=
e_{12}e_{13}-e_{12}e_{31}-e_{21}e_{13}+e_{21}e_{31}+e_{13}e_{12}-e_{13}e_{21}-e_{31}e_{12}+e_{31}e_{21}=
-e_{21}e_{13}-e_{31}e_{12}=-e_{23}-e_{32} \neq 0$.
Let $(a_{ij})=s \in S$. Then $xsy=(e_{12}-e_{21})s(e_{13}-e_{31})=-s_{23}e_{11}+s_{13}e_{21}+s_{21}e_{13}-s_{11}e_{23}$.
$e_{11}, e_{21}, e_{13}, e_{23} \in K^2$:
\bit
\item $K^2 \ni A=(e_{12}-e_{21})(e_{12}-e_{21})=e_{12}e_{12}-e_{12}e_{21}-e_{21}e_{12}+e_{21}e_{21}=-e_{12}e_{21}-e_{21}e_{12}=-e_{11}-e_{22}$.
$K^2 \ni B=(e_{13}-e_{31})(e_{13}-e_{31})=-e_{11}-e_{33}$.
$K^2 \ni C=(e_{23}-e_{32})(e_{23}-e_{32})=-e_{22}-e_{33}$.
Then, $K^2 \ni -(1/2)[A+B-C]=-(1/2)[(-e_{11}-e_{22})+(-e_{11}-e_{33})-(-e_{22}-e_{33})]=-(1/2)[-e_{11}-e_{22}-e_{11}-e_{33}+e_{22}+e_{33})]=-(1/2)[-2e_{11}]=e_{11}$.
\item $K^2 \ni -(e_{23}-e_{32})(e_{13}-e_{31})=-[e_{23}e_{13}-e_{23}e_{31}-e_{32}e_{13}+e_{32}e_{31}]=-[-e_{23}e_{31}]=e_{23}e_{31}=e_{21}$.
\item $K^2 \ni (e_{12}-e_{21})(e_{23}-e_{32})=e_{12}e_{23}-e_{12}e_{32}-e_{21}e_{23}+e_{21}e_{32}=e_{12}e_{23}=e_{13}$.
\item $K^2 \ni -[(e_{12}-e_{21})(e_{13}-e_{31})]=-[e_{12}e_{13}-e_{12}e_{31}-e_{21}e_{13}+e_{21}e_{31}]=-[-e_{21}e_{13}]=e_{23}$
\eit
\item The symplectic involution:
$\M_2(F)=K+K^2$ (which is stronger then $\M_2(F)=KS+K^2$) can be easily seen as follows, without using Theorem \ref{R simple two options for R=KS+K^2}:
\bit
\item $K^2 \ni (e_{11}-e_{22})e_{12}=e_{11}e_{12}-e_{22}e_{12}=e_{11}e_{12}=e_{12}$.
\item $K^2 \ni (e_{11}-e_{22})e_{21}=e_{11}e_{21}-e_{22}e_{21}=-e_{22}e_{21}=-e_{21}$.
\item $K^2 \ni A=(e_{11}-e_{22})(e_{11}-e_{22})=e_{11}e_{11}-e_{11}e_{22}-e_{22}e_{11}+e_{22}e_{22}=e_{11}e_{11}+e_{22}e_{22}=e_{11}+e_{22}$.
$K \ni B=(e_{11}-e_{22})$.
$K^2+K \ni A+B=(e_{11}+e_{22})+(e_{11}-e_{22})=2e_{11}$.
$K^2+K \ni A-B=(e_{11}+e_{22})-(e_{11}-e_{22})=2e_{22}$.
\eit
So, $e_{11},e_{12},e_{21},e_{22} \in K+K^2$, hence $\M_2(F)=K+K^2$.
The first condition can also be used quite easily; Take $x=e_{12} \in K$, $y=e_{21} \in K$, then $xy+yx=e_{11}+e_{22} \neq 0$. For every $S \ni s=\alpha (e_{11}+e_{22})$ $xsy=\alpha xy \in K^2 \subseteq KS+K^2$.

For $n=2m \geq 4$: Take $x=e_{1,m+1} \in K$, $y=e_{m+1,1} \in K$.
$xy-yx=e_{1,m+1}e_{m+1,1}-e_{m+1,1}e_{1,m+1}=e_{11}-e_{m+1,m+1} \neq 0$.
Let $(a_{ij})=s \in S$. Then $xsy=e_{1,m+1}se_{m+1,1}=a_{m+1,m+1}e_{11}$.
$e_{11} \in K^2$: $K^2 \ni e_{1,m+1}e_{m+1,1}=e_{11}$.
\eit
\item [(2)] The second condition says: If there exist $x \in S$ and $y \in K$ such that $xy-yx \neq 0$ and $xSy \subseteq SK+K^2$, then $SK+K^2=R$.
Matrices (with $n >2$) satisfy a stronger condition, namely: there exist $x \in S$ and $y \in K$ such that $xy-yx \neq 0$ and $xSy \subseteq K^2$.
\bit
\item The transpose involution: (Similar computations to the ones seen a few lines above show that $\M_2(F)=SK+K^2$).
For $n \geq 3$: Take $x=e_{12}+e_{21} \in S$, $y=e_{1n}-e_{n1} \in K$.
 $xy-yx=(e_{12}+e_{21})(e_{1n}-e_{n1})-(e_{1n}-e_{n1})(e_{12}+e_{21})=e_{12}e_{1n}-e_{12}e_{n1}+e_{21}e_{1n}-e_{21}e_{n1}-e_{1n}e_{12}-e_{1n}e_{21}+e_{n1}e_{12}+e_{n1}e_{21}=e_{21}e_{1n}+e_{n1}e_{12}=e_{2n}+e_{n2} \neq 0$.
Let $(a_{ij})=s \in S$. Then $xsy=(e_{12}+e_{21})s(e_{1n}-e_{n1})=-a_{2n}e_{11}-a_{1n}e_{21}+a_{21}e_{1n}+a_{11}e_{2n}$.
$e_{11}, e_{21}, e_{1n}, e_{2n} \in K^2$:
\bit
\item We have just seen that $e_{11},e_{21} \in K^2$.
\item $K^2 \ni (e_{1,n-1}-e_{n-1,1})(e_{n-1,n}-e_{n,n-1})=e_{1,n-1}e_{n-1,n}-e_{1,n-1}e_{n,n-1}-e_{n-1,1}e_{n-1,n}+e_{n-1,1}e_{n,n-1}=e_{1,n-1}e_{n-1,n}=e_{1n}$.
\item $K^2 \ni -[(e_{12}-e_{21})(e_{1n}-e_{n1})]=-[e_{12}e_{1n}-e_{12}e_{n1}-e_{21}e_{1n}+e_{21}e_{n1}]=-[-e_{21}e_{1n}]=e_{2n}$
\eit
\item The symplectic involution: 
We have already seen a few lines above that a direct computation yields $\M_2(F)=K+K^2$.
However, the second condition cannot be used, since there are no $x \in S$ and $y \in K$ such that $xy-yx \neq 0$ (indeed, $S \ni x=\beta (e_{11}+e_{22})$, so $xy-yx=\beta (e_{11}+e_{22})y-y \beta (e_{11}+e_{22})=\beta (y-y)=0$).

For $n=2m \geq 4$: Take $x=e_{1,m}+e_{2m,m+1} \in S$, $y=e_{m,2m} \in K$.
$xy-yx=(e_{1,m}+e_{2m,m+1})e_{m,2m}-e_{m,2m}(e_{1,m}+e_{2m,m+1})=e_{1,m}e_{m,2m}+e_{2m,m+1}e_{m,2m}-e_{m,2m}e_{1,m}-e_{m,2m}e_{2m,m+1}=e_{1,m}e_{m,2m}-e_{m,2m}e_{2m,m+1}=e_{1,2m}-e_{m,m+1} \neq 0$.
Let $(a_{ij})=s \in S$. Then $xsy=(e_{1,m}+e_{2m,m+1})se_{m,2m}=a_{m,m}e_{1,2m}+a_{m+1,m}e_{2m,2m}$.
$e_{1,2m}, e_{2m,2m} \in K^2$: \bit
\item $K^2 \ni (e_{1,m}+e_{2m,m+1})e_{m,2m}=e_{1,m}e_{m,2m}+e_{2m,m+1}e_{m,2m}=e_{1,2m}$.
\item $K^2 \ni e_{2m,m}e_{m,2m}=e_{2m,2m}$.
\eit
\eit
\eit
\end{example}

\subsubsection{Results about \texorpdfstring{$K+K^2, K+K^2+K^3, K+K^3$}{K+K2, K+K2+K3, K+K3}}

For the next theorem about $K+K^2$ we assume that $K$ is not a skew-commutative set.
We have seen that in a simple ring $R$ :\bit
\item If $K$ is not skew-commutative, then $R=K+KSK$.
\item If there exist $x,y \in K$ such that $xy+yx \neq 0$ and $xSy \subseteq K^2$ (in particular, $K$ is not skew- commutative), then $KS+K^2=R$ and $SK+K^2=R$.
(\item Without assuming that $K$ is not a skew-commutative set:
If there exist $x \in K, y \in S$ such that $xy-yx \neq 0$ and $xSy \subseteq KS+K^2$, then $KS+K^2=R$).
\eit
In order to show Herstein's stronger result, namely $R=K+K^2$, we demand a lot more then just non skew-commutativity of $K$;
We demand that there exist $x,y \in K$ such that $xy+yx$ is right invertible and $xSy \subseteq K^2$.
Notice that we demand much more then usual; Usually (take a look also in the next subsection), we demand that either $xy\pm yx \neq 0$ and $xAy \subseteq B$ (where $A \in \{S, K\}$ and $B \in \{S^2, K^2, KS+K^2, KS, K+KS+SK\}$) or
$xy\pm yx$ is right invertible, but not both.
We will not assume that $R$ is simple (when $xy \pm yx$ is right invertible, simplicity of $R$ is not needed).
We demand that $xy+yx$ will be right invertible (instead of $R$ is simple) since from $((xy+yx)r)=w \in K+K^2$ we are not able to show that one of the four expressions $\{wu+uw, wu-uw, wu+u^*w, wu-u^*w \}$ is also in $K+K^2$ (although for $w=((xy+yx)r) \in K+K^2$ it is true that $wu-uw \in KS+K^2$, but this is not good enough for us).
If we write $w=w_1+w_2$ where $w_1 \in K, w_2 \in K^2$, then $w_1u+u^*w_1 \in K$ and $w_2u-uw_2 \in K^2$ (we have already seen that $K^2$ is a Lie ideal of $R$, see the proof of Theorem \ref{R simple two options for R=KS+K^2}).
However, we need the same expression for both $w_1$ and $w_2$ (in order to get that the two-sided ideal $R(xy+yx)R \subseteq K+K^2$).

\begin{thm}\label{K+K^2}
Let $R$ be an associative unital $F$-algebra (not necessarily simple) with $\Char(F)\neq 2$. Let $*$ be an involution on $R$ of the first kind and $S \subsetneq R$. Assume that $K$ is not a skew-commutative set. Also assume that there exist $x,y \in K$ such that $xy+yx$ is right invertible and $xSy \subseteq K^2$ (or more generally, $xTy \subseteq K^2$), then $K+K^2=R$.
\end{thm}

\begin{proof}
$K$ is not a skew-commutative set, hence there exist $x,y \in K$ such that $xy \neq -yx$. Hence $xy+yx \neq 0$.
For every $r \in R$, $rxy+xyr^*=rxy-yxr^*+yxr^*+xyr^*=rxy-yxr^*+(xy+yx)r^*$.
Let $\alpha=\alpha(r)=rxy+xyr^*$, $\beta=\beta(r)=rxy-yxr^*$. We will just write $\alpha$ and $\beta$, instead of $\alpha(r)$ and $\beta(r)$. So for every $r \in R$, $\alpha=\beta+(xy+yx)r^*$.
Claim 1: $\alpha \in K^2$, $\beta \in K$.
Indeed, obviously $\beta \in K$, since $(\beta)^*=(rxy-yxr^*)*=yxr^*-rxy=-(rxy-yxr^*)=-\beta$.
As for $\alpha$: $\alpha=rxy+xyr^*=(rx+xr^*)y+x(yr^*+ry)-xr^*y-xry=(rx+xr^*)y+x(yr^*+ry)-x(r^*+r)y$.
Clearly, $(rx+xr^*), y, x, (yr^*+ry) \in K$. $-x(r^*+r)y \in K^2$ by our assumption that $xSy \subseteq K^2$.
Therefore, $\alpha=rxy+xyr^*=\ldots=(rx+xr^*)y+x(yr^*+ry)-x(r^*+r)y \in K^2$.
Using the claim we get: for every $r \in R$, $(xy+yx)r^*=-\beta+\alpha \in K+K^2$. But $R=R^*$, so for every $r \in R$, $(xy+yx)r \in K+K^2$.
Therefore, $(xy+yx)R \subseteq K+K^2$.
Now, $xy+yx$ is right invertible, so there exists $z \in R$ such that $(xy+yx)z=1$.
For $r \in R$, $r=1r=((xy+yx)z)r=(xy+yx)(zr)$. We have just seen that $(xy+yx)R \subseteq K+K^2$, so $r=\ldots=(xy+yx)(zr) \in K+K^2$.
So each element $r \in R$ is in $K+K^2$, hence $K+K^2=R$.
\end{proof}

\begin{example}\label{example K+K^2}

Let $R=\M_2(F)$, $\Char(F) \neq neq 2$.
The transpose involution is not an example, since for any $x,y \in K$, $xSy$ is not contained in $K^2$ ($K^2=\{\lambda(e_{11}+e_{22})| \lambda \in F\}$).
The symplectic involution is an example: Take $x=e_{12}$, $y=e_{21}$.
Then $xy+yx=e_{12}e_{21}+e_{21}e_{12}=e_{11}+e_{22}$ which is, of course, invertible.
$xSy \subseteq K^2$ (clearly, $S=\{\lambda(e_{11}+e_{22})| \lambda \in F\}$): $e_{12}(ae_{11}+ae_{22})e_{21}=
e_{12}(aI)e_{21}=a(e_{12}e_{21})=ae_{11}=a(xy) \in K^2$.
\end{example}

We can also consider $K+K^2+K^3$ and $K+K^3$.
Now we assume that $K$ is not a commutative set in order to get another theorem similar to that of Herstein's.
\begin{thm}\label{R simple xy-yx neq 0 then R=K+K^3 etc}
Let $R$ be an associative unital $F$-algebra with $\Char(F)\neq 2$. Let $*$ be an involution on $R$ of the first kind and $S \subsetneq R$. Assume that $K$ is not a commutative set. Then,
\bit
\item If there exist $x,y \in K$ such that $xy-yx$ is right invertible, then $R=K+K^2+K^3$.
\item If $R$ is simple and $K^2 \subseteq K+K^3$, then $K+K^3=R$ (now we do not demand that there exist $x,y \in K$ such that $xy-yx$ is right invertible, we just demand that $K$ is a noncommutative set).
\eit
\end{thm}
Notice that, as in Theorem \ref{K+K^2}, simplicity of $R$ is not needed when $xy-yx$ is right invertible (and not just $xy-yx \neq 0$).

\begin{proof}
$K$ is not a commutative set, hence there exist $x,y \in K$ such that $xy-yx \neq 0$.
For every $r \in R$, $rxy-xyr^*=rxy-yxr^*+yxr^*-xyr^*=rxy-yxr^*-(xy-yx)r^*$.
Let $\alpha=\alpha(r)=rxy-xyr^*$, $\beta=\beta(r)=rxy-yxr^*$. We will just write $\alpha$ and $\beta$, instead of $\alpha(r)$ and $\beta(r)$. So for every $r \in R$, $\alpha=\beta-(xy-yx)r^*$.

Claim 1: $\alpha \in K^2+K^3$, $\beta \in K$.
Indeed, obviously $\beta \in K$, since $(\beta)^*=(rxy-yxr^*)*=yxr^*-rxy=-(rxy-yxr^*)=-\beta$.
As for $\alpha$: $\alpha=rxy-xyr^*=(rx+xr^*)y-x(yr^*+ry)-xr^*y+xry=(rx+xr^*)y-x(yr^*+ry)-x(r^*-r)y$. Clearly, $(rx+xr^*), y, -x, (yr^*+ry) \in K$ and $-x(r^*-r)y \in K^3$.
Therefore, $\alpha=rxy-xyr^*=\ldots=(rx+xr^*)y-x(yr^*+ry)-x(r^*-r)y \in K^2+K^3$.
Using the claim we get: for every $r \in R$, $-(xy-yx)r^*=-\beta+\alpha \in K+K^2+K^3$. But $R=R^*$, so for every $r \in R$, $(xy-yx)r \in K+K^2+K^3$.
Therefore, $(xy-yx)R \subseteq K+K^2+K^3$.
\bit
\item If those $x,y \in K$ are such that $xy-yx$ is right invertible, then for every $r \in R$ $r=1r=((xy-yx)z)r=(xy-yx)(zr) \in K+K^2+K^3$ (we assumed that $z$ is the right inverse of $xy-yx$).
Hence, $R=K+K^2+K^3$.
\item We have just seen that $(xy-yx)R \subseteq K+K^2+K^3$. So, with our assumption that $K^2 \subseteq K+K^3$, we get that $(xy-yx)R \subseteq K+K^3$.

Claim 2: For every $u \in R$ and $w \in K+K^3$, $wu+u^*w \in K+K^3$.
Clearly, it is enough to show that for every $u \in R$, $w_1 \in K$ and $w_2 \in K^3$: $w_1u+u^*w_1 \in K$ and $w_2u+u^*w_2 \in K^3$ (since if $w=w_1+w_2$ with $w_1 \in K$ and $w_2 \in K^3$,
$wu+u^*w=(w_1+w_2)u+u^*(w_1+w_2)=w_1u+w_2u+u^*w_1+u^*w_2=(w_1u+u^*w_1)+(w_2u+u^*w_2) \in K+K^3$).
$(w_1u+u^*w_1)^*=u^*w_1^*+w_1^*u=u^*(-w_1)+(-w_1)u=-(w_1u+u^*w_1)$, so $w_1u+u^*w_1 \in K$.
As for $w_2 \in K^3$, it is enough to consider $b=xyz$ where $x,y,z \in K$.
$w_2u+u^*w_2=(xyz)u+u^*(xyz)=xy(zu+u^*z)+(u^*x+xu)yz-xyu^*z-xuyz=xy(zu+u^*z)+(u^*x+xu)yz-x(yu^*+uy)z$.
$(zu+u^*z),(u^*x+xu),(yu^*+uy) \in K$, so $w_2u+u^*w_2=(xyz)u+u^*(xyz)=\ldots=xy(zu+u^*z)+(u^*x+xu)yz-x(yu^*+uy)z \in K^3$.
Next, for every $r \in R$, we have seen that $(xy-yx)r \in K+K^3$, so from the claim just seen (with $w=(xy-yx)r$), for every $u \in R$ and for every $r \in R$, $((xy-yx)r)u+u^*((xy-yx)r) \in K+K^3$.
Use again $(xy-yx)R \subseteq K+K^3$ to get that (for every $u \in R$ and for every $r \in R$) $((xy-yx)r)u=(xy-yx)ru \in K+K^3$.
Therefore, for every $u \in R$ and for every $r \in R$, $u^*((xy-yx)r) \in K+K^3$. So $R(xy-yx)R \subseteq K+K^3$.
By assumption, $xy-yx \neq 0$, so $R(xy-yx)R$ is a non-zero two-sided ideal of $R$ ($0 \neq xy-yx=1(xy-yx)1 \in R(xy-yx)R$).
$R$ is simple, so $R(xy-yx)R=R$. But we have just seen that $R(xy-yx)R \subseteq K+K^3$, hence $R=K+K^3$.
\eit
\end{proof}
\subsection{Results about \texorpdfstring{$SKS$}{SKS}}

We had two lemmas, Lemma \ref{first lemma (xy-yx)r in S^3} and Lemma \ref{second lemma r(xy-yx)t in S^3}.
Remember that if we assume that $S$ is not a commutative set, 
those two lemmas immediately imply Theorem \ref{two sided in S^3} 
(and that theorem in turn implies its corollary, Corollary \ref{simple S^3=R}).
Here we bring two similar lemmas to the ones just mentioned, and also if we assume that $S$ is not a commutative set, then those two lemmas immediately imply a similar theorem to Theorem \ref{two sided in S^3}.
Notice that $SKS$ contains $K, SK, KS$ (since $1 \in Z(R) \subseteq S$).

\begin{lem}\label{first lemma (xy-yx)r in SKS. Again x,y in S}
Let $R$ be an associative unital $F$-algebra with $\Char(F)\neq 2$. Let $*$ be an involution on $R$ of the first kind and $S \subsetneq R$. Let $x,y \in S$. Then for every $r \in R$, $(xy-yx)r \in SKS$. In other words, the right ideal $(xy-yx)R$ is contained in $SKS$.
\end{lem}

\begin{proof}
There are two options: \bit
\item $xy-yx=0$: Then, trivially, for every $r \in R$ $(xy-yx)r=0 \in SKS$.
\item $xy-yx\neq 0$: For every $r\in R$, $rxy-xyr^*=rxy-yxr^*+yxr^*-xyr^*$.
Let $\alpha=\alpha(r)=rxy-xyr^*$, $\beta=\beta(r)=rxy-yxr^*$. We will just write $\alpha$ and $\beta$, instead of $\alpha(r)$ and $\beta(r)$. So for every $r \in R$, $\alpha=\beta-(xy-yx)r^*$.

Claim: $\alpha \in SKS$, $\beta \in K$.
Indeed, obviously $\beta \in K$, since $(\beta)^*=(rxy-yxr^*)*=yxr^*-rxy=-(rxy-yxr^*)=-(\beta)$.
As for $\alpha$: $\alpha=rxy-xyr^*=(rx-xr^*)y-x(yr^*-ry)+xr^*y-xry=(rx-xr^*)y-x(yr^*-ry)-x(-r^*+r)y$.
$(rx-xr^*), (yr^*-ry) \in K$, so $(rx-xr^*)y \in KS$ and $-x(yr^*-ry) \in SK$. $-x(-r^*+r)y \in SKS$.
Hence, $\alpha=rxy-xyr^*=\ldots=(rx-xr^*)y-x(yr^*-ry)-x(-r^*+r)y \in SKS$.
Using the claim we get: for every $r \in R$, $-(xy-yx)r^*=-\beta+\alpha \in SKS$. But $R=-R^*$, so for every $r \in R$,
$(xy-yx)r \in SKS$.
(Of course, from the first place we could take $r^*$ instead of $r$, and get, without using $R=-R^*$, that $(xy-yx)r \in SKS$).
Therefore, $(xy-yx)R \subseteq SKS$.
\eit
\end{proof}

If $S$ is a noncommutative set (in Lemma \ref{first lemma (xy-yx)r in SKS. Again x,y in S} we have not demanded that $S$ is a noncommutative set), then by definition exist $x,y \in S$ such that $xy-yx \neq 0$, therefore $SKS$ contains a non-zero right ideal of $R$, namely $(xy-yx)R$ where $x,y \in S$ with $xy-yx\neq 0$ (clearly, $0 \neq (xy-yx)1 \in (xy-yx)R$).

\begin{lem}\label{second lemma r(xy-yx)t in SKS. Again x,y in S}
Let $R$ be an associative unital $F$-algebra with $\Char(F)\neq 2$. Let $*$ be an involution on $R$ of the first kind and $S \subsetneq R$. Let $x,y \in S$. Then for every $r,u \in R$, $u(xy-yx)r \in SKS$. In other words, the two-sided ideal $R(xy-yx)R$ is contained in $SKS$.
\end{lem}

\begin{proof}
There are two options:
\bit
\item $xy-yx=0$: Then, trivially, for every $r,u \in R$ $u(xy-yx)r=0 \in SKS$.
\item $xy-yx\neq 0$:

Claim: for every $u \in R$ and $w \in SKS$, $wu+u^*w \in SKS$.
It is enough to show that for every $a,c \in S$ and $b \in K$, $abcu+u^*abc \in SKS$
(indeed, if for every $a_i,c_i \in S$ and $b_i \in K$, $a_ib_ic_iu+u^*a_ib_ic_i \in SKS$,
then taking $SKS \ni w=\sum a_ib_ic_i$, we get $wu+u^*w=(\sum a_ib_ic_i)u+u^*(\sum a_ib_ic_i)=\sum (a_ib_ic_iu+u^*a_ib_ic_i)$, so since each $a_ib_ic_iu+u^*a_ib_ic_i \in SKS$, we get $wu+u^*w=\ldots=\sum (a_ib_ic_iu+u^*a_ib_ic_i) \in SKS$).
Now, $abcu+u^*abc=ab(cu+u^*c)+(u^*a+au)bc-abu^*c-aubc=ab(cu+u^*c)+(u^*a+au)bc-a(bu^*+ub)c$.
Clearly, $cu+u^*c, u^*a+au \in S$ and $bu^*+ub \in K$, concluding that $abcu+u^*abc=\ldots=ab(cu+u^*c)+(u^*a+au)bc-a(bu^*+ub)c \in SKS$.
So we have proved the claim that for every $u \in R$ and $w \in SKS$, $wu+u^*w \in SKS$.
Next, the above lemma, Lemma \ref{first lemma (xy-yx)r in SKS. Again x,y in S}, says that for every $r \in R$,
$(xy-yx)r \in SKS$, so from the claim just seen (with $w=(xy-yx)r$), for every $u \in R$ and for every $r \in R$,
$((xy-yx)r)u+u^*(xy-yx)r \in SKS$.
Use again the above lemma to get that $((xy-yx)r)u=(xy-yx)ru \in SKS$.
Therefore, for every $u \in R$ and for every $r \in R$, $u^*(xy-yx)r in SKS$. Obviously, since $R=R^*$ we get that
for every $u \in R$ and for every $r \in R$, $u(xy-yx)r in SKS$. So $R(xy-yx)R \subseteq SKS$.
\eit
\end{proof}

If $S$ is a noncommutative set (in Lemma \ref{second lemma r(xy-yx)t in SKS. Again x,y in S} we have not demanded that $S$ is a noncommutative set), then by definition exist $x,y \in S$ such that $xy-yx \neq 0$, therefore $SKS$ contains a non-zero two-sided ideal of $R$, namely $R(xy-yx)R$ where $x,y \in S$ with $xy-yx\neq 0$.

We proceed to a theorem similar to Theorem \ref{two sided in S^3}.

\begin{thm}\label{two sided in SKS}
Let $R$ be an associative unital $F$-algebra with $\Char(F)\neq 2$. Let $*$ be an involution on $R$ of the first kind and $S \subsetneq R$. Assume $S$ is not a commutative set. Then there exists a non-zero two-sided ideal of $R$ which is contained in $SKS$.
\end{thm}

\begin{proof}
Follows immediately form Lemma \ref{second lemma r(xy-yx)t in SKS. Again x,y in S}, as was explained immediately after it.
\end{proof}

Theorem \ref{two sided in SKS} immediately implies the following:

\begin{cor}\label{simple SKS=R}
Let $R$ be a simple associative unital $F$-algebra with $\Char(F)\neq 2$. Let $*$ be an involution on $R$ of the first kind and $S \subsetneq R$. Assume $S$ is not a commutative set. Then $SKS=R$.
\end{cor}

\begin{proof}
Follows at once from Theorem \ref{two sided in SKS}.
\end{proof}

In view of Corollary \ref{cor of centS in Z}, one has:

\begin{thm}\label{SKS: simple with dim > 4}
Let $R$ be a simple associative unital $F$-algebra with $\Char(F)\neq 2$. Let $*$ be an involution on $R$ of the first kind such that $Z(R) \subsetneq S$ (and not just $Z(R) \subseteq S$) and $S \subsetneq R$. Assume that $dim_{Z(R)}R > 4$. Then $SKS=R$.
\end{thm}

\begin{proof}
By assumption, $Z(R) \subsetneq S$, so from Corollary \ref{cor of centS in Z} we get that $S$ is noncommutative.
Now Corollary \ref{simple SKS=R} implies that $SKS=R$.
\end{proof}


Remember that the first criterion for $S^2 \subseteq S^3$ to equal $R$ says the following: if there exist $x,y \in S$ such that $xy-yx \neq 0$ and for every $s \in S$ $xsy \in S^2$, then $S^2=R$.
We have seen that this condition is satisfied by matrices.
Here we consider $KS+SK \subseteq SKS$.
We show that if there exist $x,y \in S$ such that $xy-yx \neq 0$ and for every $k \in K$ $xky \in KS+SK$, then $R=KS+SK$.
However, it is already not satisfied by matrices with the transpose involution. Therefore, it is probably less reasonable for a simple ring $R$ to equal $KS+SK$ then to equal $S^2$.

\begin{lem}\label{first lemma SK+KS=R}
Let $R$ be an associative unital $F$-algebra with $\Char(F)\neq 2$. Let $*$ be an involution on $R$ of the first kind and $S \subsetneq R$. Let $x,y \in S$ such that for every $k \in K$, $xky \in KS+SK$.
Then for every $r \in R$, $(xy-yx)r \in KS+SK$. In other words, the right ideal $(xy-yx)R$ is contained in $KS+SK$.
\end{lem}

\begin{proof}
Let $x,y \in S$ such that for every $k \in K$, $xky \in KS+SK$.

There are two options: \bit
\item $xy-yx=0$: Then, trivially, for every $r \in R$ $(xy-yx)r=0 \in KS+SK$.
\item $xy-yx\neq 0$: For every $r\in R$, $rxy-xyr^*=rxy-yxr^*+yxr^*-xyr^*$.
Let $\alpha=\alpha(r)=rxy-xyr^*$, $\beta=\beta(r)=rxy-yxr^*$. We will just write $\alpha$ and $\beta$, instead of $\alpha(r)$ and $\beta(r)$. So for every $r \in R$, $\alpha=\beta-(xy-yx)r^*$.

Claim: $\alpha \in KS+SK$, $\beta \in K$.
Indeed, obviously $\beta \in K$, since $(\beta)^*=(rxy-yxr^*)*=yxr^*-rxy=-(rxy-yxr^*)=-(\beta)$.
As for $\alpha$: $\alpha=rxy-xyr^*=(rx-xr^*)y-x(yr^*-ry)+xr^*y-xry=(rx-xr^*)y-x(yr^*-ry)-x(-r^*+r)y$.
$(rx-xr^*), (yr^*-ry) \in K$, so $(rx-xr^*)y \in KS$ and $-x(yr^*-ry) \in SK$.
By assumption, $x,y \in S$ are such that for every $k \in K$, $xky \in KS+SK$, hence $-x(-r^*+r)y=x(r^*-r)y \in KS+SK$.
Hence, $\alpha=rxy-xyr^*=\ldots=(rx-xr^*)y-x(yr^*-ry)-x(-r^*+r)y \in KS+SK$.
Using the claim we get: for every $r \in R$, $-(xy-yx)r^*=-\beta+\alpha \in KS+SK$. But $R=-R^*$, so for every $r \in R$, $(xy-yx)r \in KS+SK$.
(Of course, from the first place we could take $r^*$ instead of $r$, and get, without using $R=-R^*$, that $(xy-yx)r \in KS+SK$).
Therefore, $(xy-yx)R \subseteq KS+SK$.
\eit
\end{proof}

If $S$ is a noncommutative set (in Lemma \ref{first lemma SK+KS=R} we have not demanded that $S$ is a noncommutative set), then by definition exist $x,y \in S$ such that $xy-yx \neq 0$.
However, in contrast to what we have seen above about $SKS$ (noncommutativity of $S$ implies that $SKS$ contains a non-zero right ideal of $R$), noncommutativity of $S$ does not imply that $KS+SK$ contains a non-zero right ideal of $R$.
But, if those $x,y \in S$ such that $xy-yx \neq 0$ also satisfy $xKy \subseteq KS+SK$ (this means that for every $k \in K$, $xky \in KS+SK$), then $KS+SK$ contains a non-zero right ideal of $R$ (namely, $(xy-yx)R$ where $x,y \in S$ with $xy-yx\neq 0$ and $xKy \subseteq KS+SK$).

\begin{lem}\label{second lemma SK+KS=R}
Let $R$ be an associative unital $F$-algebra with $\Char(F)\neq 2$. Let $*$ be an involution on $R$ of the first kind and $S \subsetneq R$. Let $x,y \in S$ such that for every $k \in K$, $xky \in KS+SK$. Then for every $r,u \in R$, $u(xy-yx)r \in KS+SK$. In other words, the two-sided ideal $R(xy-yx)R$ is contained in $KS+SK$.
\end{lem}

\begin{proof}
Let $x,y \in S$ such that for every $k \in K$, $xky \in KS+SK$.
There are two options: \bit
\item $xy-yx=0$: Then, trivially, for every $r,u \in R$ $u(xy-yx)r=0 \in KS+SK$.
\item $xy-yx\neq 0$:

Claim: for every $u \in R$ and $w \in KS+SK$, $wu-uw \in KS+SK$.
It is enough to show that for every $a \in S$ and $b \in K$, $abu-uab \in SK$ and for every $c \in K$ and $d \in S$,
$cdu-ucd \in KS$. Indeed, if for every $a_i,d_i \in S$ and $b_i, c_i \in K$, $a_ib_iu-ua_ib_i \in SK$ and

$c_id_iu-uc_id_i \in KS$, then taking $KS+SK \ni w=\sum_{1 \leq i \leq m} a_ib_i+c_id_i$ (we can take the same $m$

for both $\sum a_ib_i \in SK$ and $\sum c_id_i \in KS$, just add zeros to the shorter sum), we get $wu-uw=(\sum
a_ib_i+c_id_i)u-u(\sum a_ib_i+c_id_i)=\sum [(a_ib_i+c_id_i)u-u(a_ib_i+c_id_i)]=\sum
[(a_ib_i)u+(c_id_i)u-u(a_ib_i)-u(c_id_i)]=\sum [(a_ib_i)u-u(a_ib_i)+(c_id_i)u-u(c_id_i)]=\sum
[(a_ib_i)u-u(a_ib_i)]+\sum [(c_id_i)u-u(c_id_i)]$, so since $a_ib_iu-ua_ib_i \in SK$ and $c_id_iu-uc_id_i \in KS$, we
get $wu-uw=\ldots=\sum [(a_ib_i)u-u(a_ib_i)]+\sum [(c_id_i)u-u(c_id_i)] \in KS+SK$.
Now, $abu-uab=a(bu+u^*b)-(ua+au^*)b$.
Clearly, $(bu+u^*b) \in K$ and $-(ua+au^*) \in S$, concluding that $abu-uab=a(bu+u^*b)-(ua+au^*)b \in SK$.
$cdu-ucd=c(du+u^*d)-(uc+cu^*)d$.
Clearly, $(du+u^*d) \in S$ and $-(uc+cu^*) \in K$, concluding that $cdu-ucd=c(du+u^*d)-(uc+cu^*)d \in KS$.
So we have proved the claim that for every $u \in R$ and $w \in KS+SK$, $wu-uw \in KS+SK$.
Next, the above lemma, Lemma \ref{first lemma SK+KS=R}, says that for every $r \in R$, $(xy-yx)r \in KS+SK$, so from the claim just seen (with $w=(xy-yx)r$), for every $u \in R$ and for every $r \in R$, $((xy-yx)r)u-u((xy-yx)r) \in KS+SK$.
Use again the above lemma to get that $((xy-yx)r)u=(xy-yx)ru \in KS+SK$.
Therefore, for every $u \in R$ and for every $r \in R$, $u((xy-yx)r) in KS+SK$. So $R(xy-yx)R \subseteq KS+SK$.
\eit
\end{proof}

If $S$ is a noncommutative set (in Lemma \ref{second lemma SK+KS=R} we have not demanded that $S$ is a noncommutative set), then by definition exist $x,y \in S$ such that $xy-yx \neq 0$.
However, in contrast to what we have seen above about $SKS$ (noncommutativity of $S$ implies that $SKS$ contains a non-zero two-sided ideal of $R$), noncommutativity of $S$ does not imply that $KS+SK$ contains a non-zero two-sided ideal of $R$.
But, if those $x,y \in S$ such that $xy-yx \neq 0$ also satisfy $xKy \subseteq KS+SK$, then $KS+SK$ contains a non-zero two-sided ideal of $R$ (namely, $R(xy-yx)R$ where $x,y \in S$ with $xy-yx\neq 0$ and $xKy \subseteq KS+SK$).

Next, similarly to Theorem \ref{two sided in SKS} which dealt with $SKS$, we have the following theorem which deals with $KS+SK$.

\begin{thm}\label{xky in SK+KS then two sided in SK+KS}
Let $R$ be an associative unital $F$-algebra with $\Char(F)\neq 2$. Let $*$ be an involution on $R$ of the first kind and $S \subsetneq R$. Assume $S$ is not a commutative set. Also assume that there exist $x,y \in S$ such that $xy-yx \neq 0$ and for every $k \in K$, $xky \in KS+SK$. Then there exists a non-zero two-sided ideal of $R$ which is contained in $KS+SK$.
\end{thm}

\begin{proof}
Follows immediately from Lemma \ref{second lemma SK+KS=R}, as was explained immediately after it.
\end{proof}

\begin{cor}\label{simple SK+KS=R}
Let $R$ be a simple associative unital $F$-algebra with $\Char(F)\neq 2$. Let $*$ be an involution on $R$ of the first kind and $S \subsetneq R$. Assume $S$ is not a commutative set. Also assume that there exist $x,y \in S$ such that $xy-yx \neq 0$ and for every $k \in K$, $xky \in KS+SK$. Then $KS+SK=R$.
\end{cor}

\begin{proof}
Follows at once from Theorem \ref{xky in SK+KS then two sided in SK+KS}.
\end{proof}

\begin{remark}\label{SK+KS not matrix}

As was already mentioned, the condition that there exist $x,y \in S$ such that $xy-yx \neq 0$ and for every $k \in K$ $xky \in KS+SK$, is probably not satisfied by most such simple rings.
That condition is necessarily not satisfied by $R=\M_n(F)$ where $\Char(F) \neq 2$ and $n \geq 2$.
Reason: If it was satisfied, then from Corollary \ref{simple SK+KS=R} $KS+SK=\M_n(F)$.
However, $KS+SK \subsetneq \M_n(F)$; 
Indeed, One can check that $SK \subseteq \{a \in \M_n(F) | \trace(a)=0 \}$ 
(one can check that $KS \subseteq \{a \in \M_n(F) | \trace(a)=0 \}$ or one can just conclude that $KS \subseteq \{a \in \M_n(F) | \trace(a)=0 \}$ from $SK \subseteq \{a \in \M_n(F) | \trace(a)=0 \}$ and the known fact that $\trace(ab)=\trace(ba)$ for any $a,b \in \M_n(F)$, in particular one can take $a \in S, b \in K$).
Therefore, $KS+SK \subseteq \{a \in \M_n(F) | \trace(a)=0 \}$, so $KS+SK$ must be strictly contained in $\M_n(F)$.
As an exercise, one can take $n=2$, and see that for any choice of $x,y \in S, k \in K$ it is impossible to have $xy-yx \neq 0$ and $xky \in KS+SK \subseteq \{a \in \M_n(F) | \trace(a)=0 \}$.
\end{remark}

\subsection{Results about \texorpdfstring{$S^2K, KS^2$}{S2K, KS2}}

Here, in order to have a theorem for a simple ring, we are not willing to find $x,y \in S$ such that $xy \pm yx \neq 0$ (or $x,y \in K$ such that $xy \pm yx \neq 0$), instead we wish to find $s \in S, k \in K$ such that $sk-ks \neq 0$.
So it seems that commutativity or noncommutativity of $S$ will not concern us.
However it will, as the following lemma shows:

\begin{lem}\label{lemma for S^2K and KS^2}
Let $R$ be a simple associative unital $F$-algebra with $\Char(F)\neq 2$. Let $*$ be an involution on $R$ of the first kind and $S \subsetneq R$. Also assume that $dim_{Z(R)}R > 4$.
Then: $S$ is a noncommutative set $\Leftrightarrow$ there exist $s \in S, k \in K$ such that $sk-ks \neq 0$.
\end{lem}

\begin{proof}
\bit
\item Assume $S$ is a noncommutative set. If for every $s \in S, k \in K$ $sk-ks=0$, then $K \subseteq \Cent(S)$.
But the theorem of Herstein, Theorem \ref{centS in Z}, says that $\Cent(S)\subseteq Z(R)$. Therefore, $K \subseteq Z(R)$. $*$ is of the first kind ($Z(R) \subseteq S$), so $K \subseteq S$.
Hence, from $S \cap K=0$ (since $\Char(F)\neq 2$) we get $K=0$, which implies that $R=S$, a contradiction to our usual assumption that $S \subsetneq R$.
Therefore, there exist $s \in S, k \in K$ such that $sk-ks \neq 0$.
\item Assume $S$ is a commutative set. Then Corollary \ref{cor of centS in Z} shows that $S=Z(R)$, hence there is no chance to find $s \in S, k \in K$ such that $sk-ks \neq 0$.
\eit
\end{proof}

The following lemmas and theorems are very similar to ones seen thus far. But notice that now, instead of taking $x,y \in S$ we will take $x \in S$ and $y \in K$.

\begin{lem}\label{first lemma (sk-ks)r in S^2K}
Let $R$ be an associative unital $F$-algebra with $\Char(F)\neq 2$. Let $*$ be an involution on $R$ of the first kind and $S \subsetneq R$. Let $x \in S$ and $y \in K$. Then for every $r \in R$, $(xy-yx)r \in S^2K$. In other words, the right ideal $(xy-yx)R$ is contained in $S^2K$.
\end{lem}

\begin{proof}
Let $x \in S$ and $y \in K$.
There are two options: \bit
\item $xy-yx=0$: Then, trivially, for every $r \in R$ $(xy-yx)r=0 \in S^2K$.
\item $xy-yx\neq 0$: For every $r\in R$, $rxy+xyr^*=rxy+yxr^*-yxr^*+xyr^*$.
Let $\alpha=\alpha(r)=rxy+xyr^*$, $\beta=\beta(r)=rxy+yxr^*$. We will just write $\alpha$ and $\beta$, instead of $\alpha(r)$ and $\beta(r)$. So for every $r \in R$, $\alpha=\beta+(xy-yx)r^*$.

Claim: $\alpha, \beta \in S^2K$.
Indeed, obviously $\beta \in K \subseteq S^2K$, since $(\beta)^*=(rxy+yxr^*)*=-yxr^*-rxy=-(rxy+yxr^*)=-(\beta)$.
As for $\alpha$: $\alpha=rxy+xyr^*=(rx+xr^*)y+x(yr^*+ry)-xr^*y-xry=(rx+xr^*)y+x(yr^*+ry)-x(r^*+r)y$.
$(rx+xr^*) \in S$, $(yr^*+ry) \in K$, so $(rx+xr^*)y \in SK \subseteq S^2K$ and $x(yr^*+ry) \in SK \subseteq S^2K$. $-x(r^*+r)y \in SSK=S^2K$.
Hence, $\alpha=rxy+xyr^*=\ldots=(rx+xr^*)y+x(yr^*+ry)-x(r^*+r)y \in S^2K$.
Using the claim we get: for every $r \in R$, $(xy-yx)r^*=-\beta+\alpha \in S^2$. But $R=-R^*$, so for every $r \in R$, $(xy-yx)r \in S^2K$.
Therefore, $(xy-yx)R \subseteq S^2K$.
\eit
\end{proof}

If there exist $x \in S$ and $y \in K$ such that $xy-yx \neq 0$ (in Lemma \ref{first lemma (sk-ks)r in S^2K} we have not demanded this), then $S^2K$ contains a non-zero right ideal of $R$, namely $(xy-yx)R$ where $x \in S$ and $y \in K$ such that $xy-yx \neq 0$ (clearly, $0 \neq (xy-yx)1 \in (xy-yx)R$).

\begin{lem}\label{second lemma u(sk-ks)r in S^2K}
Let $R$ be an associative unital $F$-algebra with $\Char(F)\neq 2$. Let $*$ be an involution on $R$ of the first kind and $S \subsetneq R$. Let $x \in S$ and $y \in K$. Then for every $r,u \in R$, $u(xy-yx)r \in S^2K$. In other words, the two-sided ideal $R(xy-yx)R$ is contained in $S^2K$.
\end{lem}

\begin{proof}
Let $x \in S$ and $y \in K$.
There are two options: \bit
\item $xy-yx=0$: Then, trivially, for every $r,u \in R$ $u(xy-yx)r=0 \in S^2K$.
\item $xy-yx\neq 0$:

Claim: for every $u \in R$ and $w \in S^2K$, $wu+u^*w \in S^2K$.
It is enough to show that for every $a,b \in S$ and $c \in K$, $abcu+u^*abc \in S^2K$ (indeed, if for every $a_i,b_i \in S$ and $c_i \in K$, $a_ib_ic_iu+u^*a_ib_ic_i \in S^2K$, then taking $S^2K \ni w=\sum a_ib_ic_i$, we get
$wu+u^*w=(\sum a_ib_ic_i)u+u^*(\sum a_ib_ic_i)=\sum (a_ib_ic_iu+u^*a_ib_ic_i)$, so since each
$a_ib_ic_iu+u^*a_ib_ic_i \in S^2K$, we get $wu+u^*w=\ldots=\sum (a_ib_ic_iu+u^*a_ib_ic_i) \in S^2K$).
Now, $abcu+u^*abc=ab(cu+u^*c)+(u^*a+au)bc-abu^*c-aubc=ab(cu+u^*c)+(u^*a+au)bc-a(bu^*+ub)c$.
Clearly, $cu+u^*c \in K$, $u^*a+au \in S$ and $bu^*+ub \in S$, concluding that $abcu+u^*abc=\ldots=ab(cu+u^*c)+(u^*a+au)bc-a(bu^*+ub)c \in S^2K$.
So we have proved the claim that for every $u \in R$ and $w \in S^2K$, $wu+u^*w \in S^2K$.

Next, the above lemma, Lemma \ref{first lemma (sk-ks)r in S^2K}, says that for every $r \in R$, $(xy-yx)r \in S^2K$,
so from the claim just seen (with $w=(xy-yx)r$), for every $u \in R$ and for every $r \in R$,
$((xy-yx)r)u+u^*(xy-yx)r \in S^2K$.
Use again the above lemma to get that $((xy-yx)r)u=(xy-yx)ru \in S^2K$.
Therefore, for every $u \in R$ and for every $r \in R$, $u^*(xy-yx)r in S^2K$. Obviously, since $R=R^*$ we get that
for every $u \in R$ and for every $r \in R$, $u(xy-yx)r in S^2K$. So $R(xy-yx)R \subseteq S^2K$.
\eit
\end{proof}

If there exist $x \in S$ and $y \in K$ such that $xy-yx \neq 0$ (in Lemma \ref{second lemma u(sk-ks)r in S^2K} we have not demanded this), then $S^2K$ contains a non-zero two-sided ideal of $R$, namely $R(xy-yx)R$ where $x \in S$ and $y \in K$ such that $xy-yx \neq 0$ (clearly, $0 \neq 1(xy-yx)1 \in R(xy-yx)R$).

We proceed to a theorem similar to Theorem \ref{two sided in S^3}.

\begin{thm}\label{two sided in S^2K}
Let $R$ be an associative unital $F$-algebra with $\Char(F)\neq 2$. Let $*$ be an involution on $R$ of the first kind and $S \subsetneq R$. Assume that there exist $x \in S$ and $y \in K$ such that $xy-yx \neq 0$. Then there exists a non-zero two-sided ideal of $R$ which is contained in $S^2K$.
\end{thm}

\begin{proof}
Follows immediately form Lemma \ref{second lemma u(sk-ks)r in S^2K}, as was explained immediately after it.
\end{proof}

Theorem \ref{two sided in S^2K} immediately implies the following:

\begin{cor}\label{simple S^2K=R}
Let $R$ be a simple associative unital $F$-algebra with $\Char(F)\neq 2$. Let $*$ be an involution on $R$ of the first kind and $S \subsetneq R$. Assume that there exist $x \in S$ and $y \in K$ such that $xy-yx \neq 0$. Then $S^2K=R$.
\end{cor}

\begin{proof}
Follows at once from Theorem \ref{two sided in S^2K}.
\end{proof}

One can use Corollary \ref{cor of centS in Z} and Lemma \ref{lemma for S^2K and KS^2} in order to get the following theorem:

\begin{thm}\label{S^2K: simple with dim > 4}
Let $R$ be a simple associative unital $F$-algebra with $\Char(F)\neq 2$. Let $*$ be an involution on $R$ of the first kind such that $Z(R) \subsetneq S$ (and not just $Z(R) \subseteq S$) and $S \subsetneq R$. Assume that $dim_{Z(R)}R > 4$. Then $S^2K=R$.
\end{thm}

\begin{proof}
By assumption, $Z(R) \subsetneq S$, so from Corollary \ref{cor of centS in Z} we get that $S$ is a noncommutative set.
Then Lemma \ref{lemma for S^2K and KS^2} says that there exist $x \in S$ and $y \in K$ such that $xy-yx \neq 0$.
Finally, Corollary \ref{simple S^2K=R} implies that $S^2K=R$.
\end{proof}


In view of Remark \ref{SK+KS not matrix} which showed that with respect to the transpose involution, $KS+SK \subsetneq \M_n(F)$ where $\Char(F) \neq 2$ and $n \geq 2$, it seems not interesting to bring a detailed discussion for the following theorem. It is not difficult to complete all the details from Lemma \ref{first lemma (sk-ks)r in S^2K} and from Lemma \ref{second lemma u(sk-ks)r in S^2K}.

\begin{thm}\label{SK=R}
Let $R$ be a simple associative unital $F$-algebra with $\Char(F)\neq 2$. Let $*$ be an involution on $R$ of the first kind and $S \subsetneq R$. Assume that there exist $x \in S$ and $y \in K$ such that $xy-yx \neq 0$ and $xSy \subseteq SK$ (or more generally, $xTy \subseteq SK$). Then $SK=R$.
\end{thm}

Finally, only a minor change in the discussion about $S^2K$ (and $SK$) is needed in order to yield analog results for $KS^2$ (and $KS$); Just take now $x \in K$ and $y \in S$ (instead of $x \in S$ and $y \in K$).

\bibliographystyle{plain}

\end{document}